\documentclass[10pt]{amsart}
\usepackage{array,youngtab}
\usepackage{color, graphicx, comment}

\DeclareMathOperator{\modsymb}{mod}
\newcommand{\tcore}{{\text{$t$-core}}}
\newcommand{\la}{\lambda}
\newcommand{\setP}{\mathcal{P}}
\newcommand{\setH}{\mathcal{H}}
\newcommand{\setC}{\mathcal{C}}
\newcommand{\setHt}{\mathcal{H}_t}
\newcommand{\setDD}{\mathcal{DD}}
\newcommand{\setSC}{\mathcal{SC}}
\newcommand{\setW}{\mathcal{W}}
\newcommand{\opDs}{{\bar D}}
\newcommand{\sla}{{\bar\lambda}}
\newcommand{\smu}{{\bar\mu}}
\newcommand{\sal}{{\bar\sal}}
\newcommand{\twolines}[2]{{{\substack{#1\\#2}}}}
\allowdisplaybreaks

\numberwithin{equation}{section} \overfullrule 5pt
\newtheorem{thm}{Theorem}[section]
\newtheorem{cor}[thm]{Corollary}
\newtheorem{lem}[thm]{Lemma}

\theoremstyle{definition}

\title[Polynomiality of some hook-content summations]{Polynomiality of some hook-content summations for doubled distinct  and self-conjugate partitions }
\date{December 25, 2015}
\author{Guo-Niu HAN and Huan XIONG}
\address{I.R.M.A., UMR 7501, Universit\'e de Strasbourg
et CNRS, 7 rue Ren\'e Descartes, F-67084 Strasbourg, France}
\email{guoniu.han@unistra.fr}

\address{I-Math, Universit\"at Z\"urich,
  Winterthurerstrasse 190, Z\"urich 8057,
Switzerland} \email{huan.xiong@math.uzh.ch}



\begin{document}
\begin{abstract} 
In 2009, the first author proved
the Nekrasov-Okounkov formula on hook lengths for integer partitions
by using an identity of Macdonald in the framework of type $\widetilde A$ affine root systems,
and conjectured that some summations over the set of all partitions of size $n$
are
always polynomials in $n$.
This conjecture  was generalized and
proved by Stanley.
Recently, P\'etr\'eolle derived two Nekrasov-Okounkov type formulas for $\widetilde C$ and $\widetilde C\,\check{}$ which involve doubled distinct and self-conjugate partitions. 
Inspired by all those previous works, we establish 
the
polynomiality of
some hook-content summations for doubled distinct  and self-conjugate 
partitions.
\end{abstract}

\maketitle

\section{Introduction} \label{sec:introduction}  
The following so-called Nekrasov-Okounkov formula
$$
\sum_{\la\in \setP}q^{|\la|} \prod_{h\in\setH(\la)}\bigl(1-\frac {z}{h^2}\bigr) \ =\ \prod_{k\geq 1} { (1-q^k)^{z-1}},
$$
where
$\setP$ is the set of all
integer partitions $\lambda$ with $|\lambda|$ denoting the size of  $\lambda$ and $\setH(\lambda)$ the 
multiset of hook lengths associated with $\lambda$ (see \cite{han2}),
was discovered independently several times: First, by Nekrasov and Okounkov 
in their study of the theory of Seiberg-Witten on supersymmetric gauges in
particle physics \cite{no}; Then, proved by Westbury using D'Arcais polynomials \cite{westbury}; 
Finally, by the first author using an identity of Macdonald \cite{Macdonald} in the framework of type $\widetilde A$ affine root systems~\cite{han2}.
Moreover, he asked to find Nekrasov-Okounkov type formulas 
associated with other root systems
\cite[Problem 6.4]{han},
and conjectured that
$$n! \sum_{|\lambda|= n}
\frac{1}{ H(\la)^2}\sum_{h\in\setH(\lambda)}h^{2k} $$ 
is always a
polynomial in $n$ for any $k\in \mathbb{N}$,
where
$H(\la)=\prod_{h\in\setH(\la)} h$.
This conjecture  was 
proved by Stanley in a more general form. In particular, he showed that
$$n! \sum_{|\lambda|= n}
\frac{1}{ H(\la)^2}\ F_1(h^2: h\in\setH(\lambda))\, F_2(c: c\in\setC(\lambda)) $$ 
is a
polynomial in $n$ for any symmetric functions $F_1$ and $F_2$, where
$\setC(\la)$ is the multiset of contents associated with $\la$ (see \cite{stan}).  
For some special functions $F_1$ and $F_2$ the latter polynomial has
explicit expression, as shown by  Fujii, Kanno, Moriyama, Okada and Panova \cite{fkmo,panova}.
\medskip

A \emph{strict partition} is a finite strict decreasing sequence of
positive integers $\bar\lambda = (\bar\lambda_1, \bar\lambda_2, \ldots,
\bar\lambda_\ell)$. The integer  $|\bar\lambda|=\sum_{1\leq i\leq
\ell}\bar\lambda_i$ is called the \emph{size} and $\ell(\sla)=\ell$ is called the {\it length} of $\bar\lambda.$ For convenience, let $\bar\lambda_i=0$ for $i>\ell(\bar\lambda)$.   A strict partition $\bar\lambda$ could be identical with
its shifted Young diagram, which means that the $i$-th row of the
usual Young diagram is shifted to the right by $i$ boxes. 
We define the \emph{doubled distinct partition} 
of $\bar\la$, denoted by $\bar\la\bar\la$,
 to be the usual partition
whose Young diagram is obtained by adding $\bar\la_i$ boxes to the $i$-th column of
the shifted Young diagram of $\bar\la$
for $1 \leq i \leq \ell(\bar\la)$ (see \cite{stanton,petreolle1,petreolle2}). For example, $(6,4,4,1,1)$ is the doubled distinct partition of $(5,2,1)$ (see  Figure 1).

\begin{figure}
\centering
\begin{center}
\includegraphics[]{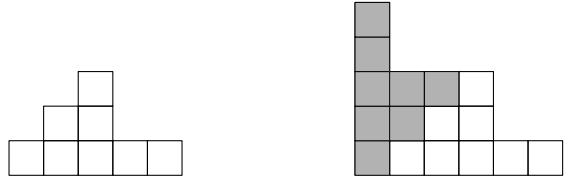}
\end{center}
\caption{From strict partitions  to doubled distinct partitions.}
\end{figure}

For each usual partition $\la$, let $\la'$ denote the conjugate partition of $\la$ (see \cite{stanton,Macdonald,petreolle1,petreolle2}).
A usual partition $\la$ is called \emph{self-conjugate} if $\la=\la'$.
The set of all doubled distinct partitions and the set of all self-conjugate partitions
are denoted by $\setDD$ and $\setSC$ respectively.
For each positive integer $t$,
let
$$
{\mathcal{H}}_t(\lambda)=\{ h\in\mathcal{H}(\lambda): \ h\equiv 0 \pmod t\}
$$
be  the multiset
of  hook lengths of multiples of $t$.
Write $H_t(\la)=\prod_{h\in\setH_t(\la)} h$.

Recently, P\'etr\'eolle derived two Nekrasov-Okounkov type formulas for $\widetilde C$ and $\widetilde C\,\check{}$ which involve doubled distinct and self-conjugate partitions. 
In particular, he obtained the following two formulas \cite{petreolle1,petreolle2}.

\begin{thm}[P\'etr\'eolle \cite{petreolle1,petreolle2}]\label{th:petreolle}
	For positive integers $n$ and $t$ we have
\begin{align}
	\sum_\twolines{\la\in\setDD, |\la|=2nt}{\#\setHt(\la)=2n}
	\quad
	\frac 1{H_t(\la)} = \frac{1}{(2t)^n n!},\quad \text{if $t$ is odd};
	\label{eq:DDt}\\
	\sum_\twolines{\la\in\setSC, |\la|=2nt}{\#\setHt(\la)=2n}
	\quad
	\frac 1{H_t(\la)} = \frac{1}{(2t)^n n!},\quad \text{if $t$ is even}.\label{eq:SCt}
\end{align}
\end{thm}

Inspired by all those previous works, we establish 
the
polynomiality of
some hook-content summations for doubled distinct  and self-conjugate 
partitions. Our main result is stated next.

\begin{thm}\label{th:intro}
Let $t$ be a given positive integer. The following two summations for the positive integer $n$
\begin{equation}\label{eq:DD}
(2t)^n n!	\sum_\twolines{\la\in\setDD, |\la|=2nt}{\#\setHt(\la)=2n}
	\quad
	\frac {F_1(h^2: h\in\setH(\la))\,F_2(c: c\in\setC(\la))}{H_t(\la)}
	\quad \text{$($$t$ odd\,$)$}
\end{equation}
and 
\begin{equation}\label{eq:SC}
(2t)^n n!	\sum_\twolines{\la\in\setSC, |\la|=2nt}{\#\setHt(\la)=2n}
	\quad
	\frac {F_1(h^2: h\in\setH(\la))\,F_2(c: c\in\setC(\la))}{H_t(\la)}
	\quad \text{$($$t$ even\,$)$}
\end{equation}
are polynomials in $n$ for any symmetric functions $F_1$ and $F_2$.
\end{thm}

In fact, the degrees of the two polynomials in Theorem \ref{th:intro} 
can be estimated explicitly  in terms of $F_1$ and $F_2$ (see Corollary \ref{th:main''} and Theorem \ref{th:SCmain3}).
When $F_1$ and $F_2$ are two constant symmetric functions, we derive Theorem \ref{th:petreolle}.
Other specializations are listed as follows.
\begin{cor} \label{th:DDSCsquare}
 We have
\begin{align}
(2t)^n n!\sum\limits_{\substack{\lambda \in \mathcal{DD}, 
|\lambda|=2nt \\ \#\mathcal{H}_t(\lambda)=2n} } 
\frac{1}{H_t(\la)}\, \sum\limits_{\substack{h\in \setH(\lambda)}}h^{2}
&=6t^2 n^2+\frac 13 (t^2-6t+2) tn \quad \text{$($$t$ odd\,$)$},\label{eq:DDh2}\\
(2t)^n n!\sum\limits_{\substack{\lambda \in \mathcal{SC}, 
|\lambda|=2nt \\ \#\mathcal{H}_t(\lambda)=2n} } 
\frac{1}{H_t(\la)}\, \sum\limits_{\substack{h\in \setH(\lambda)}}h^{2}
&=6t^2 n^2+\frac 13 (t^2-6t-1) tn \quad \text{$($$t$ even\,$)$},\label{eq:SCh2}\\
(2t)^n n!\sum\limits_{\substack{\lambda \in \mathcal{DD}, 
|\lambda|=2nt \\ \#\mathcal{H}_t(\lambda)=2n} } 
\frac{1}{H_t(\la)}\, \sum\limits_{\substack{c\in \setC(\lambda)}}c^{2}
&=2t^2 n^2+\frac 13 (t^2-6t+2) tn \quad \text{$($$t$ odd\,$)$},\label{eq:DDc2}\\
(2t)^n n!\sum\limits_{\substack{\lambda \in \mathcal{SC}, 
|\lambda|=2nt \\ \#\mathcal{H}_t(\lambda)=2n} } 
\frac{1}{H_t(\la)}\, \sum\limits_{\substack{c\in \setC(\lambda)}}c^{2}
&=2t^2 n^2+\frac 13 (t^2-6t-1) tn \quad \text{$($$t$ even\,$)$}.\label{eq:SCc2}
\end{align}
\end{cor}

The rest of the paper is essentially devoted to complete the proof of  Theorem~\ref{th:intro}. 
The polynomiality of summations in \eqref{eq:DD} for $t=1$ with $F_1=1$ or $F_2=1$ has an equivalent statement in terms of strict partitions, whose proof is given in Section 2. 
After recalling some basic definitions and properties of Littlewood decomposition in Section 3, the doubled distinct and self-conjugate cases of Theorem \ref{th:intro} are proved in Sections 4 and 5 respectively.
Finally, Corollary \ref{th:DDSCsquare} is proved in Section 6.


\section{Polynomiality for strict and doubled distinct  partitions} \label{sec:main}  

In this section we prove an equivalent statement  of the polynomiality of \eqref{eq:DD} for $t=1$ with $F_1=1$ or $F_2=1$, which consists a summation over the set of strict partitions.
Let $\sla=(\sla_1, \sla_2, \ldots, \sla_\ell)$ be a strict partition. Therefore the leftmost box in the $i$-th row of the shifted Young diagram of $\sla$
has coordinate $(i,i+1)$. 
The \emph{hook length} of 
the $(i, j)$-box, 
denoted by $h_{(i, j)}$, is defined to be  the number of boxes exactly to the
right, or  exactly above, or the box itself, plus~$\sla_{j}$.
For example, consider the box $\square=(i,j)=(1,3)$ in the shifted Young diagram of the strict partition $(7,5,4,1)$.
There are 1 and 5 boxes below and to the right of the box $\square$ respectively. Since $\sla_3=4$, the hook length of $\square$ is equal to $1+5+1+4=11$, as illustrated in Figure~1.
The {\it content} of $\square=(i,j)$ is defined to be $c_\square=j-i$,
so that the leftmost box in each row has content $1$.
Also, 
let $\mathcal{H}(\sla)$ be the multi-set of hook lengths of boxes
and
$H({\sla})$ be the product of all hook lengths of boxes in~$\sla$.
The hook length and content multisets of the doubled distinct partition 
$\sla\sla$ can be 
obtained from $\setH(\sla)$ and $\setC(\sla)$ by the following relations:
\begin{align}
	\setH(\sla\sla)&=
\setH(\sla) \cup\setH(\sla) \cup \{2\sla_1, 2\sla_2, \ldots, 2\sla_{\ell}\} 
\setminus \{\sla_1, \sla_2, \ldots, \sla_\ell\},\label{eq:Hlala}\\
\setC(\sla\sla)&= \setC(\sla) \cup \{1-c \mid c\in \setC(\sla)\}. \label{eq:Clala}
\end{align}

\medskip

\begin{figure}
\centering
\begin{center}
\includegraphics[]{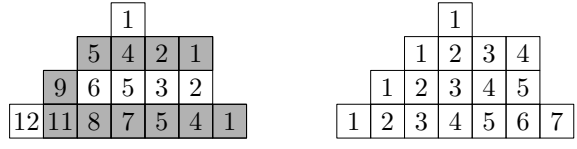}
\end{center}
\caption{The shifted
 Young diagram, the hook lengths and the contents  of the strict partition $(7,5,4,1)$.}
\end{figure}

For two strict partitions $\sla$ and $\smu$, we write $\sla \supseteq \smu$
if $\sla_i\geq \smu_i$ for any $i\geq 1$.
In this case,  the skew strict partition $\sla/\smu$ is identical with the
skew shifted Young diagram. For example, the skew strict partition $(7,5,4,1)/(4,2,1)$ is represented by the white boxes in Figure 2.
Let $f_\sla$ (resp.
$f_{\sla/\smu}$) be the number of standard shifted Young tableaux of shape
$\sla$ (resp. $\sla/\smu$).  The following formulas for strict partitions are well-known (see \cite{bandlow, schur,  thr}):
\begin{equation}
	\label{eq:hookformula*} f_\sla = \frac{|\sla| !}{H(\sla)}
\qquad \text{and}\qquad \frac{1}{n!}\sum_{|\sla| =n}
2^{n-\ell(\sla)}f_\sla^2=1.
\end{equation}
Identity \eqref{eq:DDt} with $t=1$, obtained by P\'etr\'eolle, becomes
$$
\sum_{\la\in\setDD, |\la|=2n}
	\frac 1{H(\la)} = \frac{1}{2^n n!},
$$
which is equivalent to the second identity of \eqref{eq:hookformula*} in view of \eqref{eq:Hlala}. 

\medskip

\begin{figure}
\centering
\begin{center}
\includegraphics[]{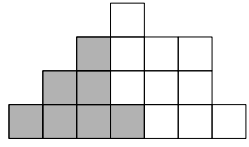}
\end{center}
\caption{The skew shifted Young diagram of the skew strict partition $(7,5,4,1)/(4,2,1)$.}
\end{figure}

\begin{figure}
\centering
\begin{center}
\includegraphics[]{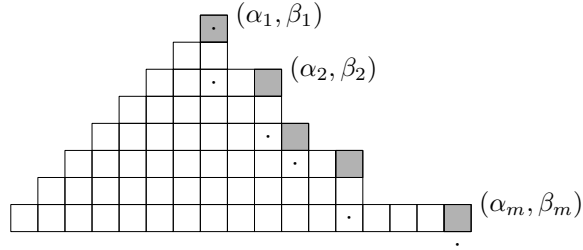}
\end{center}
\caption{A strict partition and its corners. The outer corners are labelled with
$(\alpha_i, \beta_i)$ ($i=1,2,\ldots, m$). The inner corners are indicated by
the dot symbol ``$\cdot$''.}
\end{figure}

For a strict partition $\sla$, the \emph{outer corners} (see 
\cite{hanxiong3}) are the boxes which can be removed 
in such a way that after removal the resulting diagram is still a shifted Young diagram of a strict partition.
The coordinates of outer corners are denoted by
$(\alpha_1,\beta_1),\ldots,(\alpha_{m},\beta_{m})$  such that
$\alpha_1>\alpha_2>\cdots
>\alpha_m$. Let $y_j:=\beta_j-\alpha_j$ ($1\leq j \leq m$) be the contents of outer corners. We set $\alpha_{m+1}=0,\
\beta_0=\ell(\sla)+1$ and call
$(\alpha_1,\beta_0),(\alpha_2,\beta_1),\ldots,(\alpha_{m+1},\beta_{m})$
 the  \emph{inner corners} of $\sla$. Let $x_i=\beta_i-\alpha_{i+1}$
 be the contents of inner corners for $0\leq i \leq m$ (see Figure 3).
The following relation of $x_i$ and $y_j$ are obvious.
\begin{equation}
x_0=1\leq y_1<x_1<y_2<x_2<\cdots <y_m <x_m.
\end{equation}
Notice that $x_0= y_1=1$ iff $\sla_{\ell(\sla)}=1$.
Let $\sla^{i+}=\sla\bigcup \{ \square_i\}$ such that
$c_{\square_i}=x_i$ for $0\leq i\leq m$. 
Here $\sla^{0+}$ does not exist if $y_1=1$.
The set of contents of inner corners and the set of contents of outer corners of $\sla$ are 
denoted by $X(\sla)=\{x_0,x_1,\ldots, x_m\}$ and $Y(\sla)=\{y_1, y_2, \ldots,y_m\}$ respectively. 
The following relations between the hook lengths of $\sla$ and $\sla^{i+}$
are established in \cite{hanxiong3}.

\begin{thm}[Theorem 3.1 of \cite{hanxiong3}] \label{th:hanxiong3}
Let $\sla$ be a strict partition with
$X(\sla)=\{x_0,x_1,\ldots, x_m\}$ and $Y(\sla)=\{y_1, y_2, \ldots,y_m\}$. 
For $1\leq i \leq m$, we have
\begin{align*}
	&\mathcal{H}(\sla)
	\cup
	 \{ 1,x_i, 2x_i-2 \} \cup  
\{ |x_i-x_j|:1\leq j \leq m, j\neq i \} \\ 
& \qquad\qquad \cup \{x_i+x_j-1:1\leq j \leq m, j\neq i \}  \\
=\,&
	\mathcal{H}(\sla^{i+})
\cup
\{ |x_i-y_j|:1\leq j \leq m \} \cup \{ x_i+y_j-1:1\leq j \leq m \}
\end{align*}
and
\begin{equation*}
	\frac{H(\sla)}{H(\sla^{i+})}=\frac12\cdot\frac{\prod\limits_{\substack{1\leq
j\leq m}}\bigl(\binom{x_i}{2}-\binom{y_j}{2}\bigr)}{\prod\limits_{\substack{0\leq j\leq m\\
j\neq i}}\bigl(\binom{x_i}{2}-\binom{x_j}{2}\bigr)}.
\end{equation*}
If $y_1>1$, we have
 \begin{align*}
 &\mathcal{H}(\sla) \cup
 \{ 1,x_1,x_1-1,x_2,x_2-1,\cdots, x_m,x_m-1 \} \\
 =\,&
\mathcal{H}(\sla^{0+})\ \cup 
\{ y_1,y_1-1,y_2,y_2-1,\cdots, y_m,y_m-1 \}
\end{align*}
and
\begin{equation*}
	\frac{H({\sla})}{H({\sla^{0+}})}=\frac{\prod\limits_{\substack{1\leq
j\leq
m}}\bigl(\binom{x_0}{2}-\binom{y_j}{2}\bigr)}{\prod\limits_{\substack{1\leq
j\leq m}}\bigl(\binom{x_0}{2}-\binom{x_j}{2}\bigr)}.
\end{equation*} 
\end{thm}

\medskip

Let $k$ be a nonnegative integer, 
and 
$\nu=(\nu_1, \nu_2, \ldots, \nu_{\ell(\nu)})$ be a usual partition. 
For arbitrary two finite alphabets $A$ and $B$, 
the power sum of the alphabet $A-B$ is defined by \cite[p.5]{Lascoux}
\begin{align}
	\Psi^k(A, B)&:=\sum_{a\in A}
{a}^{k}-\sum_{b\in B} {{b}}^{k},\\
\Psi^\nu(A, B)&:=\prod_{j=1}^{\ell(\nu)}\Psi^{\nu_j}(A,B).\label{def:Psi}
\end{align}
Let $\sla$ be a strict partition.  We define
\begin{equation}\label{def:qk}
\Phi^\nu(\sla):=\Psi^\nu 
\Bigl(	\{\binom{x_i}{2}\},  \{\binom{y_i}{2}\} \Bigr).
\end{equation}

\begin{thm}[Theorem 3.5 of \cite{hanxiong3}] \label{th:diffq1}
Let $k$ be a given nonnegative integer. Then,
 there exist some $\xi_j\in \mathbb{Q}$ such that
$$
\Phi^k(\sla^{i+})-\Phi^k(\sla)=\sum_{j=0}^{k-1}\xi_j{\binom{x_i}{2}}^j
$$
for every strict partition $\sla$ and $0\leq i \leq m$,
where $x_0, x_1, \ldots, x_m$ are the contents of inner corners of $\sla$.
\end{thm}
\begin{lem}\label{x2}
Let $k$ be a given nonnegative integer. Then, there exist some $a_{ij}$ such that 
$$(x-y)^{2k}+(x+y-1)^{2k}=\sum_{i+j\leq k}a_{ij}\binom{x}{2}^i\binom{y}{2}^j$$ 
for every $x,y\in \mathbb{C}$.
\end{lem}
\begin{proof}
The claim follows from
$$(x-y)^2+(x+y-1)^2=4\binom{x}{2}+4\binom{y}{2}+1$$
and
\begin{equation*}
	(x-y)^2(x+y-1)^2=\bigl(2\binom{x}{2}-2\binom{y}{2}\bigr)^2.\qedhere
\end{equation*}
\end{proof}

\begin{lem}[Theorem 3.2 of \cite{hanxiong3}] \label{th:aibi}
Let $k$ be a nonnegative  integer. Then, there exist some
$\xi_{\nu}\in \mathbb{Q}$ indexed by usual partitions $\nu$ such that
\begin{equation*} \sum_{0\leq i\leq m}
\frac{\prod\limits_{\substack{1\leq
j\leq m}}({a_i}-{{b}_j})}{\prod\limits_{\substack{0\leq j\leq m\\
j\neq i}}({a_i}-{a_j})}a_i^k
=\sum_{|\nu|\leq k}\xi_{\nu}\Psi^\nu(\{a_i\},\{{b}_i\})
\end{equation*}
for arbitrary  complex numbers $a_0<a_1<\cdots<a_m$ and
${b}_1<{b}_2<\cdots<{b}_m$. 
\end{lem}

We define
{\it the difference operator $\opDs$ for strict partitions}  by
\begin{equation}
\opDs\bigl(g(\sla)\bigr):=
	2\sum_{i=1}^mg(\sla^{i+}) + g(\sla^{0+}) -\ g(\sla), 
\end{equation}
where $\sla$ is a strict partition and $g$ is a function of strict partitions. 
In the above definition, the symbol $g(\sla^{0+})$ takes the value $0$ 
if $\sla^{0+}$ does not exist, or equivalently if $\sla_{\ell{(\sla)}} = 1$.
By Theorem~\ref{th:hanxiong3}, we have
\begin{equation}\label{eq:DH=0}			
	\opDs\Bigl(\frac{1}{H({\sla})}\Bigr)=0.
\end{equation}

\begin{thm}[Theorem 2.3 of \cite{hanxiong3}]\label{th:stricttelescope}
	Let $g$ be a function of strict partitions and $\smu$ be a given
	strict partition. Then we have
	\begin{equation}\label{eq:telescope*}
		\sum_{
		|\sla/\smu|=n}
		2^{|\sla|-|\smu|-\ell(\sla)+\ell(\smu)}f_{\sla/\smu}g(\sla)
		=\sum_{k=0}^n\binom{n}{k}\opDs^kg(\smu)
	\end{equation}
	and
	\begin{equation}\label{eq:telescope**}
		\opDs^ng(\smu)=\sum_{k=0}^n(-1)^{n+k}\binom{n}{k}
		\sum_{|\sla/\smu|=k}2^{|\sla|-|\smu|-\ell(\sla)+\ell(\smu)}f_{\sla/\smu}g(\sla).
	\end{equation}

	\goodbreak

	In particular, if there exists some positive integer $r$ such that
	$\opDs^r g(\sla)=0$ for every strict partition $\sla$, then the
	left-hand side of \eqref{eq:telescope*} is a polynomial of $n$ with degree at most $r-1$.
\end{thm}

For each usual partition $\delta$ let
$$
p^\delta(\sla):=\Psi^\delta(\{h^2:h\in \setH(\sla\sla)\}, \emptyset).
$$
By \eqref{eq:Hlala}, we have
$$p^{k}(\sla)
=
\sum_{h\in \setH(\sla\sla)}h^{2k}
=
2\sum_{h\in \setH(\sla)}h^{2k}+(4^k-1)\sum_{i=1}^{\ell(\sla)}\sla_i^{2k}
$$
for a nonnegative integer $k$.

\begin{thm} \label{th:strictmain}
Suppose that  $\nu$ and $\delta$ are two given usual partitions. 
Then, 
\begin{equation}\label{eq:Drstrict}
\opDs^r\Bigl( \frac{p^\delta(\sla)\Phi^{\nu}(\sla)}{H(\sla)} \Bigr)=0
\end{equation}
for every strict partition $\sla$, where $r=|\delta|+\ell(\delta)+|\nu|+1$.
Consequently, for a given strict partition $\smu$,
\begin{equation}\label{eq:strictpoly}
\sum_{|\sla/\smu|=n}\frac{2^{|\sla|-\ell(\sla)}f_{\sla/\smu}}{H(\sla)}p^\delta(\sla)
\end{equation}
is a polynomial in $n$ of degree at most $|\delta|+\ell(\delta)$.
\end{thm}

\begin{proof}
Let $X(\sla)=\{x_0,x_1,\ldots, x_m\}$ and $Y(\sla)=\{y_1, y_2, \ldots,y_m\}$. 
First, we show that the difference $p^k(\sla^{i+}) - p^k(\sla)$ can be written as the following form
\begin{equation*}
\sum_{j=0}^k\eta_{j}(\sla)\binom{x_i}{2}^j
\end{equation*}
for $0\leq i\leq m$ and a nonnegative integer $k$, where each coefficient $\eta_j(\sla)$ is a linear combination of some
$\Phi^{\tau}(\sla) $ for some usual partition $\tau$ of size $|\tau|\leq k$.
Indeed, by Lemma \ref{x2} and  Theorem  \ref{th:hanxiong3}, 
\begin{align*}
\ \ \ & \ \ \ p^k(\sla^{0+})-p^k(\sla)=2\sum_{j=1}^m(x_j^{2k}+(x_j-1)^{2k})-2\sum_{j=1}^m(y_j^{2k}+(y_j-1)^{2k})+2^{2k}+1\\&=\eta_0(\sla)= \sum_{j=0}^k\eta_{j}(\sla)\binom{x_0}{2}^j
	\qquad\qquad \text{[\, if $i=0$ and $\sla_{\ell(\sla)}\geq 2$\,]}
\end{align*}
and 
\begin{align*}
\ \ \ & \ \ \ p^k(\sla^{i+})-p^k(\sla)\\&=2\sum_{j=1}^m((x_i-x_j)^{2k}+(x_i+x_j-1)^{2k})-2\sum_{j=1}^m\bigl((x_i-y_j)^{2k}+(x_i+y_j-1)^{2k}\bigr)\\&+2x_i^{2k}+2{(2x_i-2)}^{2k}+2-2{(2x_i-1)}^{2k}+(2^{2k}-1)\bigl(x_i^{2k}-(x_i-1)^{2k}\bigr)
\\&= \sum_{j=0}^k\eta_{j}(\sla)\binom{x_i}{2}^j
\ \qquad\qquad\qquad \text{[\, if $1\leq i\leq m$ \,]}.
\end{align*}

\goodbreak

Next, let $A= \Phi^\nu(\sla)$ and $B= p^\delta(\sla)$.
We have
\begin{align*}
	\Delta_i A
 &  := \Phi^\nu(\sla^{i+}) -\Phi^\nu(\sla) = \sum\limits_{(*)}\prod_{s\in U}\Phi^{\nu_s}(\sla)
{\prod_{s'\in V} \bigl(\Phi^{\nu_{s'}}(\sla^{i+})-\Phi^{\nu_{s'}}(\sla)\bigr)},\\
	\Delta_i B
&:= p^\delta(\sla^{i+}) -p^\delta(\sla)= \sum\limits_{(**)}\prod_{s\in U}p^{\delta_s}(\sla)
{\prod_{s'\in V} \bigl(p^{\delta_{s'}}(\sla^{i+})-p^{\delta_{s'}}(\sla)\bigr)},
\end{align*}
where the sum $(*)$ (resp. $(**)$) ranges over all
pairs $(U,V)$ of positive integer sets such that 
$U\cup V=\{1,2,\ldots,\ell(\nu)\}$
(resp. $U\cup V=\{1,2,\ldots,\ell(\delta)\}$), 
$U\cap V=\emptyset$ and $V\neq \emptyset$.

Finally, it follows from 
 \eqref{eq:DH=0} and
Theorem
 \ref{th:hanxiong3}
that
\begin{align*}
 \ &\ \ \ \ \ H(\sla) \opDs\Bigl( \frac{p^{\delta}(\sla)\Phi^{\nu}(\sla)}{H(\sla)}
\Bigr)
\\
&=
\frac{H(\sla)}{H{(\sla^{0+})}} \bigl(p^{\delta}(\sla^{0+})\Phi^{\nu}(\sla^{0+})-
p^{\delta}(\sla)\Phi^{\nu}(\sla)\bigr)
\\
&
\quad\qquad
+2\sum_{i=1}^m\frac{H(\sla)}{H({\sla^{i+}})}\bigl(p^{\delta}(\sla^{i+})\Phi^{\nu}(\sla^{i+})-
p^{\delta}(\sla)\Phi^{\nu}(\sla)\bigr)
\\
&=
\sum_{0\leq i\leq m} \frac{\prod\limits_{\substack{1\leq
j\leq m}}\bigl({\binom{x_i}{2}}-{\binom{y_j}{2}}\bigr)}{\prod\limits_{\substack{0\leq j\leq m\\
j\neq i}}\bigl({\binom{x_i}{2}}-{\binom{x_j}{2}}\bigr)}\bigl(p^{\delta}(\sla^{i+})\Phi^{\nu}(\sla^{i+})-
p^{\delta}(\sla)\Phi^{\nu}(\sla)\bigr)\\
&=
\sum_{0\leq i\leq m} \frac{\prod\limits_{\substack{1\leq
j\leq m}}\bigl({\binom{x_i}{2}}-{\binom{y_j}{2}}\bigr)}{\prod\limits_{\substack{0\leq j\leq m\\
j\neq i}}\bigl({\binom{x_i}{2}}-{\binom{x_j}{2}}\bigr)}\
\bigl(
A\cdot \Delta_i B + B\cdot  \Delta_i A + \Delta_i A\cdot \Delta_i B
\bigr).
\end{align*}
By Theorems  \ref{th:aibi} and \ref{th:diffq1},
each of the above three terms could be written
as a linear combination of some $p^{\underline \delta}(\sla)
\Phi^{\underline \nu}(\sla)$ 
satisfying $|\underline \delta|+\ell(\underline \delta)+|\underline \nu|\leq |\delta|+\ell(\delta)+|\nu|-1$.
Then the claim follows by induction on $|\delta|+\ell(\delta)+|\nu|$.
\end{proof}

When $\smu=\emptyset$, 
the summation \eqref{eq:strictpoly} in Theorem \ref{th:strictmain} becomes
\begin{equation}
\sum_{|\sla|=n}\frac{2^{n-\ell(\sla)}n!}{H(\sla)^2}p^\delta(\sla)
\end{equation}
or
\begin{equation}
2^{n}n!\sum_{|\sla\sla|=2n}\frac{1}{H(\sla\sla)}
\Psi^\delta(\{h^2:h\in \setH(\sla\la)\}, \emptyset)
\end{equation}
by \eqref{eq:Hlala}.
The above summation is a polynomial in $n$.
Consequently, Theorem \ref{th:intro} is true when $t=1$ and $F_2=1$.
Other specializations are listed as follows.

\begin{thm}
Let $\smu$ be a given strict partition. Then,
\begin{equation}\label{eq:sum--c:k=1}
\sum_{|\sla/\smu|=n}\frac{2^{|\sla|-\ell(\sla)-|\smu|+\ell(\smu)}f_{\sla/\smu}H_\smu}{H_{\sla}}\bigl(p^1(\sla)-p^1(\smu)\bigr)
=12\binom{n}{2}+(12|\smu|+5)n.
\end{equation}
Let $\smu=\emptyset$. We obtain
\begin{equation}\label{eq:sumDD}
2^{n}n!\sum_{|\sla\sla|=2n}\frac{1}{H(\sla\sla)}
\sum_{h\in \setH(\sla\sla)}{h^2}=12\binom{n}{2}+5n.
\end{equation}

\end{thm}

\begin{proof}
We have
\begin{align*}
\ \ \ & \ \ \ p^1(\sla^{0+})-p^1(\sla)=2\sum_{j=1}^m(x_j^{2}+(x_j-1)^{2})-2\sum_{j=1}^m(y_j^{2}+(y_j-1)^{2})+2^{2}+1\\&=\eta_0(\sla)= 8|\sla|+5
	\qquad\qquad \text{[\, if $i=0$ and $\sla_{\ell(\sla)}\geq 2$\,]}
\end{align*}
and 
\begin{align*}
\ \ \ & \ \ \ p^1(\sla^{i+})-p^1(\sla)\\&=2\sum_{j=1}^m((x_i-x_j)^{2}+(x_i+x_j-1)^{2})-2\sum_{j=1}^m((x_i-y_j)^{2}+(x_i+y_j-1)^{2})\\&+2x_i^{2}+2{(2x_i-2)}^{2}+2-2{(2x_i-1)}^{2}+(2^{2}-1)(x_i^{2}-(x_i-1)^{2})
\\&= 4\binom{x_i}{2}+8|\sla|+5
\ \qquad\qquad\qquad \text{[\, if $1\leq i\leq m$ \,]}.
\end{align*}
So that
\begin{align*}
 H_{\sla}D\Bigl(\frac{p^1(\sla)}{H_{\sla}}\Bigr)
&=
\sum_{0\leq i\leq m} \frac{\prod\limits_{\substack{1\leq
j\leq m}}\bigl({\binom{x_i}{2}}-{\binom{y_j}{2}}\bigr)}{\prod\limits_{\substack{0\leq j\leq m\\
j\neq i}}\bigl({\binom{x_i}{2}}-{\binom{x_j}{2}}\bigr)}(4\binom{x_i}{2}+8|\sla|+5)
 \\&
 =4\Phi^1(\sla)+8|\sla|+5\\
 &
 =12|\sla|+5.
 \end{align*}
 Therefore we have
 \begin{align*}
H_{\sla}D^2\Bigl(\frac{p^1(\sla)}{H_{\sla}}\Bigr)
&=
12,
\\
H_{\sla}D^3\Bigl(\frac{p^1(\sla)}{H_{\sla}}\Bigr)
&=
0.
\end{align*}
Identity \eqref{eq:sum--c:k=1} follows from Theorem \ref{th:stricttelescope}. 
By \eqref{eq:Hlala}, we derive \eqref{eq:sumDD}.
\end{proof}

Recall the following results obtained in \cite{hanxiong3} involving the contents of strict partitions.

\begin{thm} \label{th:content}
Suppose that $Q$ is a given symmetric function, and $\smu$ is a given strict partition. Then
\begin{align*}
\sum_{|\sla/\smu|=n}\frac{2^{|\sla|-|\smu|-\ell(\sla)+\ell(\smu)}f_{\sla/\smu}}{H(\sla)}Q\Bigl(\binom{c}{2}:
{c}\in\setC(\sla)\Bigr)
\end{align*}
is a polynomial in $n$.
\end{thm}

\begin{thm} \label{th:content:k}
Suppose that $k$ is a given nonnegative integer. Then
\begin{align*}
\sum_{|\sla|=n}\frac{2^{|\sla|-\ell(\sla)}f_{\sla}}{H(\sla)}
\sum_{c\in\setC(\sla)}
\binom{c+k-1}{2k}=\frac{2^k}{(k+1)!}\binom{n}{k+1}.
\end{align*}
\end{thm}

\begin{thm} \label{th:content:1}
Let $\smu$ be a strict partition. Then,
\begin{equation}\label{eq:sum-c:k=1}
	\sum_{|\sla/\smu|=n}\frac{2^{|\sla|-\ell(\sla)-|\smu|+\ell(\smu)}f_{\sla/\smu}H_\smu}{H(\sla)}\bigl(\sum_{c\in\setC(\sla)}\binom{c}{2}-\sum_{c\in\setC(\smu)}
\binom{c}{2}\bigr)
=\binom{n}{2}+n|\smu|.
\end{equation}
\end{thm}

The above results can be interpreted in terms of doubled distinct partitions.
In particular, we obtain Theorem \ref{th:intro} when $t=1$ and $F_1=1$. 
\begin{thm}\label{th:cDD}
For each usual partition $\delta$, the summation 
\begin{equation}
2^{n}n!\sum_{|\sla\sla|=2n}\frac{1}{H(\sla\sla)}
	\Psi^\delta(\setC(\sla\sla), \emptyset)
\end{equation}
is a polynomial in $n$. 
\end{thm}
\begin{proof}
Since $c+(1-c)=1$ and $c(1-c)=-2\binom c2$, there exists some  $a_i$ such that $c^k+(1-c)^k=\sum_{i=1}^sa_i{\binom c2}^i$.
By \eqref{eq:Clala}, we obtain
$$
\sum_{c\in \setC(\sla\sla)}c^k
=\sum_{c\in \setC(\sla)}\bigl(c^k+(1-c)^k\bigr)
=\sum_{i=1}^sa_i\sum_{c\in \setC(\sla)}{\binom c2}^i.
$$
The claim follows from Theorem \ref{th:content}.
\end{proof}

The following results are corollaries of
Theorems \ref{th:content:k} and \ref{th:content:1}.
\begin{thm}
Suppose that $k$ is a given nonnegative integer. Then,
\begin{align}
2^{n}n!\sum_{|\sla\sla|=2n}\frac{1}{H(\sla\sla)}
\sum_{c\in\setC(\sla\sla)}
\binom{c+k-1}{2k}&=\frac{2^{k+1}}{(k+1)!}\binom{n}{k+1},\\
	2^{n}n!\sum_{|\sla\sla|=2n}\frac{1}{H(\sla\sla)}
	\sum_{c\in \setC(\sla\sla)} c^2
	&=4\binom{n}{2}+\binom n1.
\end{align}
\end{thm}


\section{The Littlewood decomposition and corners of usual partitions}\label{sec:corners}

In this section we recall 
some basic definitions and properties for usual partitions (see \cite{han4}, \cite[p.12]{Macdonald}, \cite[p.468]{stan1}, \cite[p.75]{JK}, \cite{stanton}).
Let $\setW$ be the set of bi-infinite binary sequences beginning
with infinitely many 0's and ending with infinitely many 1's. Each
element $w$ of $\setW$ can be represented by $(a'_i)_{i} =\cdots
a'_{-3}a'_{-2}a'_{-1}a'_0a'_1a'_2a'_3\cdots$. However, the representation is
not unique, since for any fixed integer $k$ the sequence
$(a'_{i+k})_i$ also represents $w$. The {\it canonical
representation} of $w$ is the unique sequence
$(a_i)_i= \cdots a_{-3}a_{-2}a_{-1}a_0a_1a_2a_3\cdots$
such that
$$
\#\{i\leq -1, a_i=1\}=\#\{i\geq 0, a_i=0\}.
$$
It will be further denoted by
$ \cdots a_{-3}a_{-2}a_{-1}.a_0a_1a_2a_3\cdots$
with a dot symbol inserted between the letters $a_{-1}$ and $a_0$.
\medskip
There is a natural one-to-one correspondence between $\setP$ and
$\setW$ (see, e.g. \cite[p.468]{stan1}, \cite{AF} for more details). Let $\la$
be a partition. We encode each horizontal edge of $\la$ by 1 and
each vertical edge by 0. Reading these (0,1)-encodings from top to
bottom and from left to right yields a binary word~$u$. By adding
infinitely many 0's to the left and infinitely many 1's to the
right of $u$ we get an element $w=\cdots 000u111\cdots\in\setW$.
Clearly, the map $\la\mapsto w$ is a one-to-one correspondence
between $\setP$ and $\setW$. For example, take
$\la=(6,3,3,1)$. Then~$u=0100110001$, so that $w=(a_i)_i=\cdots
1110100.110001000\cdots$ (see Figure~5).

\begin{figure}
\centering
\begin{center}
\includegraphics[]{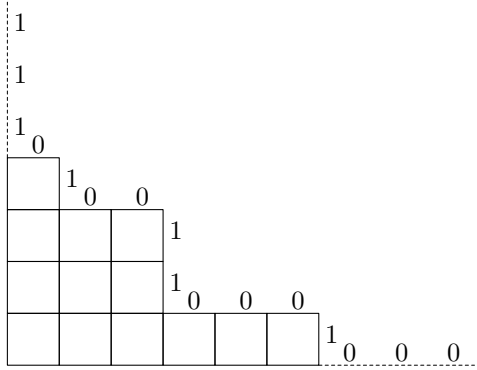}
\end{center}
\caption{From usual partitions to bi-infinite 01-sequences.}
\end{figure}
\medskip

Let $t$ be a positive integer. 
Recall that a partition $\la$ is a \emph{$t$-core} if it has no 
hook equal to $t$. The set of all $t$-core partitions (resp. $t$-core doubled distinct partitions) are denoted by $\setP_\tcore$ (resp. $\setDD_\tcore$).
The
\emph{Littlewood decomposition} maps a usual partition $\la$ to
$(\la_\tcore; \la^0, \la^1, \ldots, \la^{t-1})\in \setP_\tcore\times \setP^{t}$ such that
\smallskip

(P1) {$\la_\tcore$ is a
$t$-core and $\la^0, \la^1, \ldots, \la^{t-1}$ are usual partitions;}
\smallskip

(P2) {$|\la|=|\la_\tcore|+t(|\la^0|+|\la^1|+\cdots+|\la^{t-1}|)$;}
\smallskip

(P3) {$\{h/t\mid h\in \setH_t(\la)\}
= \setH(\la^0)\cup \setH(\la^1)\cup \cdots \cup \setH(\la^{t-1})$.}
\smallskip

\noindent The vector $(\la^0, \la^1, \ldots,
\la^{t-1})$ is called the {\it $t$-quotient} of the
partition~$\la$.
\medskip

It is well know that  (see \cite{stanton}) under the Littlewood decomposition, a doubled distinct partition $\lambda$ has image 
$(\lambda_{\tcore}; \lambda^0, \lambda^1,\ldots, \lambda^{t-1})\in \setDD_\tcore\times \setDD\times \setP^{t-1}$ where $\lambda^i$ is the conjugate partition of $\lambda^{t-i}$ for $1\leq i\leq t-1$.  

\medskip

\begin{figure}
\centering
\begin{center}
\includegraphics[]{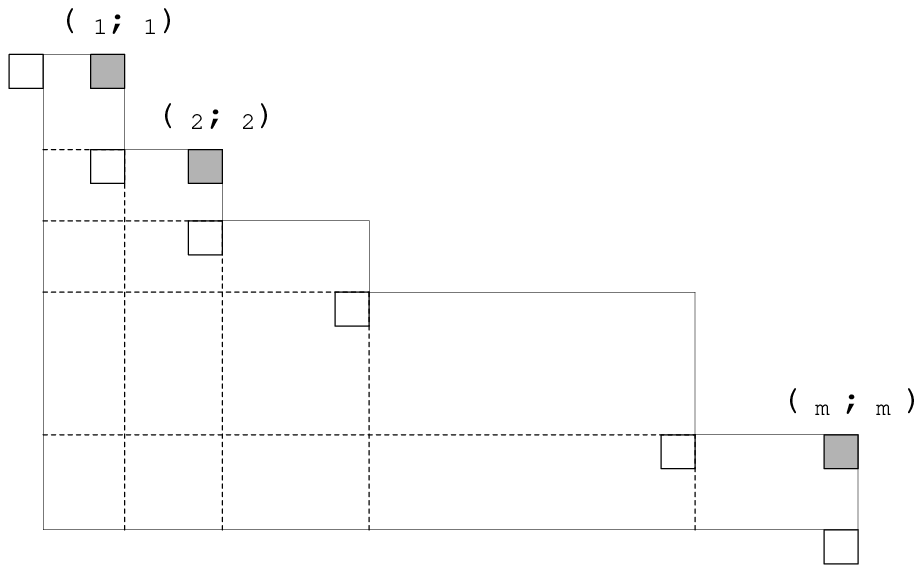}
\end{center}
\caption{A partition and its corners. The outer corners are labelled with
$(\alpha_i, \beta_i)$ ($i=1,2,\ldots, m$). The inner corners are indicated by
the dot symbol ``$\cdot$''.}
\end{figure}

For a usual partition
$\lambda$, the \emph{outer corners} (see \cite{hanxiong1, bandlow}) are the
boxes which can be removed to get a new partition. Let
$(\alpha_1,\beta_1),\ldots,(\alpha_{m},\beta_{m})$ be the
coordinates of outer corners such that $\alpha_1>\alpha_2>\cdots
>\alpha_m$. Let $y_j=\beta_j-\alpha_j$ be the contents of outer
corners for $1\leq j \leq m.$ We set $\alpha_{m+1}=\beta_0=0$ and
call
$(\alpha_1,\beta_0),(\alpha_2,\beta_1),\ldots,(\alpha_{m+1},\beta_{m})$
 the  \emph{inner corners} of $\lambda$. Let $x_i=\beta_i-\alpha_{i+1}$
 be the contents of inner corners for $0\leq i \leq m$ (see Figure 6).
It is easy to verify that $x_i$ and $y_j$ satisfy the following relation:
\begin{equation}
x_0<y_1<x_1<y_2<x_2<\cdots <y_m <x_m.
\end{equation}
Define (see \cite{hanxiong1})
\begin{equation}\label{def:qk1}
	\Psi^\nu(\lambda):=\Psi^\nu(\{x_i\}, \{y_j\})
\end{equation}  
for  each usual partition $\nu$.

\begin{lem} \label{th:XY}
Suppose that $\lambda$ is a partition whose set of contents of inner corners and set of contents of outer corners are $X(\lambda)=\{x_0,x_1,\ldots, x_m\}$ and $Y(\lambda)=\{y_1, y_2, \ldots,y_m\}$ respectively. Let $\lambda^{i+}=\lambda\cup\{\square_i\}$ where $c_{\square_i}=x_i$.
Then we have 
$$X(\lambda^{i+})\cup \{x_i,x_i\}\cup Y(\lambda)
=X(\lambda) \cup \{x_i+1,x_i-1\}\cup Y(\lambda^{i+}) .$$
\end{lem}
\begin{proof}
 Four cases are to be considered. (i) If
$\beta_{i}+1<\beta_{i+1}$ and $\alpha_{i+1}+1<\alpha_{i}$. Then,
the contents of inner corners and outer corners
of $\lambda^{i+}$ are $X\cup \{x_i-1, x_i+1\} \setminus \{x_i\}$ and
$Y\cup \{x_i\}$ respectively. (ii) If $\beta_{i}+1=\beta_{i+1}$ and
$\alpha_{i+1}+1<\alpha_{i}$, so that $y_{i+1}=x_i+1$. Hence the
contents of inner corners and outer corners of $\lambda^{i+}$ are
$X\cup \{x_i-1\} \setminus \{x_i\}$ and $Y\cup \{x_i\}\setminus
\{x_i+1\}$ respectively. (iii) If $\beta_{i}+1<\beta_{i+1}$ and
$\alpha_{i+1}+1=\alpha_{i}$, so that $y_{i}=x_i-1$. Then the
contents of inner corners and outer corners of $\lambda^{i+}$ are
$X\cup \{x_i+1\} \setminus \{x_i\}$ and $Y\cup \{x_i\}\setminus
\{x_i-1\}$ respectively. (iv) If $\beta_{i}+1=\beta_{i+1}$ and
$\alpha_{i+1}+1=\alpha_{i}$. Then $y_{i}+1=x_i=y_{i+1}-1$. The
contents of inner corners and outer corners of $\lambda^{i+}$ are $X
\setminus \{x_i\}$ and $Y\cup \{x_i\}\setminus \{x_i-1, x_i+1\}$
respectively. The claim is proved.
\end{proof}

The corners of the strict partition $\sla$ and the doubled distinct partition $\sla\sla$ are closely related.
\begin{lem} \label{th:XYstrictDD}
Suppose that $\sla$ is a strict partition whose set of contents of inner corners and set of contents of outer corners are $X(\sla)=\{x_0,x_1,\ldots, x_m\}$ and $Y(\sla)=\{y_1, y_2, \ldots,y_m\}$ respectively. Then, 
$$X(\sla\sla)\cup \{y_1,1-y_1, \ldots,y_m,1-y_m\}=Y(\sla\sla) \cup \{0,x_1,1-x_1,\ldots, x_m,1-x_m\}.$$

\end{lem}
\begin{proof}
Two cases are to be considered.
(i) If $y_1=1$, the
contents of inner corners and outer corners of $\sla\sla$ are
$X(\sla\sla)=\{x_1,1-x_1,\ldots, x_m,1-x_m\}$ and $Y(\sla\sla)=\{1,y_2,1-y_2 \ldots,y_m,1-y_m\}$ respectively. 
(ii) If $y_1\geq 2$, the
contents of inner corners and outer corners of $\sla\sla$ are
$X(\sla\sla)=\{0,x_1,1-x_1,\ldots, x_m,1-x_m\}$ and $Y(\sla\sla)=\{y_1,1-y_1, \ldots,y_m,1-y_m\}$ respectively. 
This achieves the proof of Lemma~\ref{th:XYstrictDD}.
\end{proof}


\section{The $t$-difference operators for doubled distinct partitions} 
Let $t=2t'+1$ be an odd positive integer.
For each strict partition $\sla$, 
the doubled distinct partition associated with $\sla$ is denoted by $\la=\sla\sla$. 
The Littlewood decomposition maps 
$\sla\sla$ to
$$(\lambda_{\tcore}; \lambda^0, \lambda^1,\ldots, \lambda^{2t'})\in 
\setDD_\tcore\times \setDD\times \setP^{2t'}$$ 
where $\lambda^i$ is the conjugate partition of $\lambda^{t-i}$ for $1\leq i\leq t'$.  
For convenience we say that the {\it Littlewood decomposition} maps the strict partition $\sla$ to
\begin{equation}\label{def:LittlewoodDD}
\sla \mapsto (\sla_\tcore;
\sla^0,\lambda^1,\ldots,\lambda^{t'}),
\end{equation}
where $\sla_\tcore$ and $\sla^0$ are determined by
$\la_\tcore=\sla_\tcore \sla_\tcore$ and $\la^0 = \sla^0 \sla^0$.
Since the map \eqref{def:LittlewoodDD} is bijective, we always write
\begin{equation*}
\la = (\sla_\tcore; \sla^0,\lambda^1,\ldots,\lambda^{t'}).
\end{equation*}

Let $\la=\sla\sla=(\sla_\tcore; \sla^0,\lambda^1,\ldots,\lambda^{t'})$ and $\mu=\smu\smu=(\smu_\tcore; \smu^0,\mu^1,\ldots,\mu^{t'})$ be two doubled distinct partitions.
 If  $\lambda_\tcore=\mu_\tcore$, $\sla^0\supset \smu^0$ and $\lambda^i\supset \mu^i$ for $1\leq i\leq t'$, 
we write 
$\lambda\geq_t \mu$  and define
\begin{equation}
F_{\mu/\mu}:=1
\qquad \text{and}\qquad F_{\lambda/\mu}:=\sum\limits_{\substack{\lambda\geq_t\lambda^-\geq_t\mu\\
|\lambda/\lambda^-|=2t }}F_{\lambda^-/\mu}\qquad \text{(for $\lambda \neq \mu$)}.
\end{equation}
In fact,
$F_{\lambda/\mu}$
is the number of vectors 
$(P_0,P_1,\ldots,P_{t'})$
such that

 (1) $P_0$ {is a skew shifted 
Young tableau of shape}  $\sla^0/\smu^0$, 

 (2) $P_i$ $(1\leq i\leq t')$  {is a skew
Young tableau of shape}  $\lambda^i/\mu^i$, 

(3) {the union of
entries in} $P_0,P_1,\ldots,P_{t'}$ \emph{are} $1,\ 2,\ldots,\ 
|\sla^0/\smu^0|+\sum_{i=1}^{t'}|\lambda^i/\mu^i|$.

\noindent
Hence,
$$F_{\lambda/\mu}=\binom{|\sla^0/\smu^0|+\sum_{i=1}^{t'}|\lambda^i/\mu^i|}
{|\sla^0/\smu^0|,|\lambda^1/\mu^1|,\ldots,|\lambda^{t'}/\mu^{t'}|}f_{\sla^0/\smu^0}\prod_{i=1}^{t'}
 f_{\lambda^i/\mu^i}.$$
We set
\begin{equation}
F_\lambda:=
F_{\lambda/ \lambda_{\tcore}}=
\binom{|\sla^0|+\sum_{i=1}^{t'}|\lambda^i|}
{|\sla^0|,|\lambda^1|,\ldots,|\lambda^{t'}|}f_{\sla^0}\prod_{i=1}^{t'}
 f_{\lambda^i}
 =\frac{n!}{H({\sla^0})\prod_{i=1}^{t'} H({\lambda^i})}
 \end{equation}
and 
$$G_{\lambda}:=
\frac {2^{n-\ell(\sla^0)}}{t^nH({\sla^0})\prod_{i=1}^{t'} H({\lambda^i})}
=
\frac {2^{n-\ell(\sla^0)} F_\la}{t^n n!},
$$
 where $n=|\sla^0|+\sum_{i=1}^{t'}|\lambda^i|$.

When $t=1$, we have $t'=0,$  thus $F_\lambda=f_{\sla^0}$ and 
$G_\lambda = 2^{n-\ell(\sla^0)}/H(\sla^0)$.
Also, when $\lambda$ is a $t$-core doubled distinct partition, we have $F_\lambda=G_\lambda=1$.

\subsection{$t$-difference operators}

Let $g$ be a function of doubled distinct partitions and $\lambda$ be a doubled distinct partition.
The \emph{$t$-difference operator}~$D_t$ for doubled distinct partitions is defined by 
\begin{equation}\label{eq:diffDtDD}  D_tg(\lambda)=\sum\limits_{\substack{\lambda^{+}\geq_t\lambda\\ |\lambda^{+}/\lambda|=2t}}g(\lambda^{+})-g(\lambda).
\end{equation}
The higher-order $t$-difference operators  $D_t^k$ are defined by induction:
$$D_t^0g:=g \text{\quad and\quad }\ D_t^k g:=D_t(D_t^{k-1} g)\quad (k\geq 1).$$

\begin{lem} \label{th:Glambda}
Let $\lambda$ be a doubled distinct partition. Then,
$$
D_t (G_\lambda)=0.
$$
In other words,
$$G_{\lambda}=\sum\limits_{\substack{\lambda^+\geq_t\lambda\\
|\lambda^+/\lambda|=2t }}G_{\lambda^+}.
$$
\end{lem}
\begin{proof}
Write $\lambda=(\sla_\tcore;
\sla^0,\lambda^1,\ldots,\lambda^{t'})$.
By Theorem $3.3$ in  \cite{hanxiong3} we obtain
$$
\sum_{|(\sla^0)^+/\sla^0|=1}
\frac{G_{(\sla_\tcore;\ 
(\sla^0)^+,\lambda^1,\ldots,\lambda^{t'})}}{G_\lambda}
=
\sum_{|(\sla^0)^+/\sla^0|=1}
\frac{2^{1+\ell(\sla^0)-\ell((\sla^0)^+)}H(\sla^{0})}{tH((\sla^{0})^+)}
=\frac{1}{t}.
$$
For $1\leq i\leq t'$
we have 
$$
\sum_{|(\lambda^{i})^+/\lambda^{i}|=1}
\frac{G_{(\sla_\tcore;
\sla^0,\lambda^1,\ldots,\lambda^{i-1},(\lambda^{i})^+,\lambda^{i+1},\ldots,\lambda^{t'})}}{G_\lambda}
=
\sum_{|(\lambda^{i})^+/\lambda^{i}|=1}
\frac{2H(\lambda^{i})}{tH((\lambda^{i})^+)}
=
\frac{2}{t}
$$
by Lemma 2.2 in \cite{hanxiong1}. Summing the above equalities, we get
\begin{equation*}\sum\limits_{\substack{\lambda^+\geq_t\lambda\\
|\lambda^+/\lambda|=2t }}\frac{G_{\lambda^+}}{G_\la}
=\frac 1t + \sum_{i=1}^{t'} \frac 2t = 1.\qedhere
\end{equation*}
\end{proof}

\begin{lem} \label{th:telescope}
Suppose that $\mu$ is a  given doubled distinct partition and $g$ is a function    of doubled distinct
 partitions. For every $n\in\mathbb{N}$, let
$$
P(\mu,g; n):=\sum\limits_{\substack{\la\in\setDD,\ \lambda\geq_t\mu\\
|\lambda/\mu|=2nt }}F_{\lambda/\mu}g(\lambda).
$$
Then $$P(\mu,g; n+1)-P(\mu, g; n)=P(\mu, {D_tg}; n).$$

\end{lem}
\begin{proof}
The proof is straightforward:
\begin{align*}
P(\mu,g; n+1)-P(\mu, g; n)
&=\sum\limits_{\substack{\nu\geq_t\mu\\
|\nu/\mu|=2(n+1)t }}F_{\nu/\mu}g(\nu)-\sum\limits_{\substack{\lambda\geq_t\mu\\
|\lambda/\mu|=2nt }}F_{\lambda/\mu}g(\lambda)\\
&= \sum\limits_{\substack{\nu\geq_t\mu\\
|\nu/\mu|=2(n+1)t }}\sum\limits_{\substack{\nu\geq_t\nu^-\geq_t\mu\\
|\nu/\nu^-|=2t }}F_{\nu^-/\mu}g(\nu)-\sum\limits_{\substack{\lambda\geq_t\mu\\
|\lambda/\mu|=2nt }}F_{\lambda/\mu}g(\lambda)\\
&=\sum\limits_{\substack{\lambda\geq_t\mu\\
|\lambda/\mu|=2nt }}F_{\lambda/\mu}\bigl(\sum\limits_{\substack{\lambda^{+}\geq_t\lambda\\
|\lambda^{+}/\lambda|=2t}}g(\lambda^+)-g(\lambda)\bigr) \\
&= P(\mu,D_tg; n).\qedhere
\end{align*}
\end{proof}

{\it Example}. Let  $g(\lambda)=G_\lambda$. Then $D_tg(\lambda)=0$
by Lemma \ref{th:Glambda}, which means that $P(\mu, D_tg; n)=0$.
Consequently, $P(\mu, G_\la; n+1)=P(\mu, G_\la; n)=\cdots=P(\mu, G_\la; 0)=G_\mu,$ or
\begin{equation}
\sum\limits_{\substack{\la\in\setDD,\ \lambda\geq_t\mu\\
|\lambda/\mu|=2nt }}F_{\lambda/\mu} G_\lambda= G_\mu.
\end{equation}
When $\mu$ is a $t$-core doubled distinct partition, 
the above identity becomes
\begin{equation*}
\sum\limits_{\substack{\lambda\geq_t\mu\\
|\lambda/\mu|=2nt }}
\frac{n!}{H({\sla^0})\prod_{i=1}^{t'} H({\lambda^i})} \times
\frac {2^{n-\ell(\sla^0)}}{t^nH({\sla^0})\prod_{i=1}^{t'} H({\lambda^i})}
= G_\mu, 
\end{equation*}
or
\begin{equation*}
\sum\limits_{\substack{\lambda \in \mathcal{DD}, \
|\lambda/\mu|=2nt \\ \lambda\geq_t\mu } } \frac{(2t)^n n!}
{H_t(\la)}\,=1,
\end{equation*}
which implies \eqref{eq:DDt}. 

\medskip

\begin{thm}\label{th:main3}
Let $g$ be a function of  doubled distinct partitions and $\mu$ be a given doubled distinct partition.
Then, 
\begin{equation}\label{eq:main3DD}
P(\mu, g; n)=\sum\limits_{\substack{\la\in\setDD,\ \lambda\geq_t\mu\\
|\lambda/\mu|=2nt }}F_{\lambda/\mu}g(\lambda)
=\sum_{k=0}^n\binom{n}{k}D_t^kg(\mu)
\end{equation}
and
\begin{equation}\label{eq:main3bDD}
D_t^ng(\mu)=\sum_{k=0}^n(-1)^{n+k}\binom{n}{k}P(\mu, g; k).
\end{equation}
In particular, if there exists some positive integer $r$ such that
$D_t^r g(\lambda)=0$ for every doubled distinct partition $\lambda\geq_t \mu$, then $P(\mu, g; n)$ is
a polynomial in $n$ with degree at most $r-1$.
\end{thm}

\begin{proof}
Identity \eqref{eq:main3DD} is proved by induction. The case $n=0$
is obvious. Assume that \eqref{eq:main3DD} is true for some
nonnegative integer $n$. By Lemma \ref{th:telescope} we obtain
\begin{align*}
	P(\mu, g; n+1)&= P(\mu, g; n)+ P(\mu, {D_tg}; n)
\\&=\sum_{k=0}^n\binom{n}{k}D_t^kg(\mu)+\sum_{k=0}^n\binom{n}{k}D_{t}^{k+1}g(\mu)
\\&=\sum_{k=0}^{n+1}\binom{n+1}{k}D_t^kg(\mu).
\end{align*}
Identity \eqref{eq:main3bDD} follows from the famous M\"obius inversion
formula \cite{Rota1964}.
\end{proof}

\subsection{$\mu$-admissible functions of doubled distinct partitions}\label{sec:ProofMain} 
Let $\mu=\smu\smu$ be a $t$-core doubled distinct partition.
A function $g$ of doubled distinct partitions  is called {\it $\mu$-admissible}, if for  each given $1\leq i \leq t'$ (resp. $i=0$),
$g(\lambda^+)-g(\lambda)$
is a polynomial in $c_{\square_i}$  (resp. $\binom{c_{\square_0}}{2}$)  for every pair of partitions
$$\lambda=(\smu;
\sla^0,\lambda^1,\ldots,\lambda^{t'})$$
and
$$
\lambda^+=(\smu;
\sla^0,\lambda^1,\ldots,\lambda^{i-1},\lambda^{i}\cup{\square_i},\lambda^{i+1},\ldots,\lambda^{t'})
$$
$$
(\text{resp.\ } \lambda^+=(\smu;
\sla^0\cup{\square_0},\lambda^1,\ldots,\lambda^{t'})),
$$
whose coefficients are of form
$$
\sum K(\mu,i; \tau^0, \tau^1, \ldots, \tau^{t'})\,\Phi^{\tau^0}(\sla^0)\prod_{j=1}^{t'}\Psi^{\tau^j}(\lambda^j),$$ 
where the summation is taken over the set of $(t'+1)$-tuple of usual partitions
$(\tau^0, \tau^1, \ldots, \tau^{t'})$ and $K$ is a function.

\goodbreak

\begin{lem}\label{th:admissible}
Let $\mu$ be a $t$-core doubled distinct partition. Then,
the two functions of doubled distinct partitions
$ \sum_{h\in \setH(\lambda)}h^{2r} $
and
$ \sum_{c\in \setC(\lambda)}c^{r} $ 
are $\mu$-admissible for 
any nonnegative integer~$r$. 
\end{lem}

To prove Lemma \ref{th:admissible}, we recall some results
on the multisets of hook lengths and contents, obtained in \cite{hanxiong2}. 
Suppose that a given $t$-core partition $\mu$
has 01-sequence $w(\mu)=(a_{\mu,j})_{j\in \mathbb{Z}}$.
For $0\leq i\leq t-1$ we define \cite{hanxiong2} 
$$
b_i:=b_i(\mu)=\min\{j\in \mathbb{Z}:j\equiv  i(\modsymb t),\
a_{\mu,j}=1\}.
$$


\begin{lem}[Lemma 5.3 of \cite{hanxiong2}] \label{th:contentadd1}
Let $\la$ be a partition and 
$(\la_\tcore; \la^0, \la^1, \ldots, \la^{t-1})$ be the image of the Littlewood decomposition of $\la$.  Then,
$$
\mathcal{C}(\lambda)\setminus \mathcal{C}(\la_\tcore)= \bigcup_{i=0}^{t-1}
\{tc+b_i(\la_\tcore)-j:0\leq j\leq t-1, c \in \setC(\lambda^i)\}.
$$
\end{lem}

\begin{lem}[Lemma 5.4 of \cite{hanxiong2}]\label{th:hookdiff}
Let $0\leq i \leq t-1$,
$\la$ and $\la^+$ be two usual partitions whose
images of the Littlewood decomposition are  
$(\la_\tcore; \la^0, \la^1, \ldots, \la^{t-1})$ 
and
$$
(\la_\tcore;
\lambda^0,\lambda^1,\ldots,\lambda^{i-1},\lambda^{i}\cup\{\square_i\},\lambda^{i+1},\ldots,\lambda^{t-1})
$$ 
respectively. Write $b_j=b_j(\la_\tcore)\ (0\leq j \leq t-1)$.
Suppose $r$ is a given integer, $1\leq k\leq t-1$. Let
$x_{j,s}\ (0\leq s\leq m_j)$ be the contents of inner corners of
$\lambda^j$
and
$y_{j,s}\ (1\leq s\leq m_j)$ be the  contents of outer corners of
$\lambda^j$ for $0\leq j\leq t-1$.
We have
 \begin{multline*}
\sum\limits_{\substack{\square\in \lambda^+\\ h_{\square}\equiv0
(\modsymb t )}}h_{\square}^{2r} - \sum\limits_{\substack{\square\in
\lambda\\ h_{\square}\equiv0 (\modsymb t )}}h_{\square}^{2r} =\\ t^{2r}+
\sum\limits_{\substack{0\leq s \leq
    m_i}} \left(t(c_{\square_i}-x_{i,s})\right)^{2r}-
    \sum\limits_{\substack{1\leq s \leq
    m_i}} \left(t(c_{\square_i}-y_{i,s})\right)^{2r}
\end{multline*} and
\begin{align*}
\ &\sum\limits_{\substack{\square\in\lambda^+\\ h_\square\equiv
k (\modsymb t )}}h_\square^{2r}+\sum\limits_{\substack{\square\in\lambda^+\\
h_\square\equiv t-k (\modsymb t )}}h_\square^{2r} -
\sum\limits_{\substack{\square\in\lambda\\ h_\square\equiv
k (\modsymb t )}}h_\square^{2r}-\sum\limits_{\substack{\square\in\lambda\\
h_\square\equiv t-k (\modsymb t )}}h_\square^{2r}\\ &=
\sum\limits_{\substack{0\leq s \leq
    m_{i'}}} (tc_{\square_i}+b_i-tx_{i',s}-b_{i'})^{2r}-
    \sum\limits_{\substack{1\leq s \leq
    m_{i'}}} (tc_{\square_i}+b_i-ty_{i',s}-b_{i'})^{2r}
    \\ &+ \sum\limits_{\substack{0\leq
s \leq
    m_{i''}}} (tc_{\square_i}+b_i-tx_{i'',s}-b_{i''})^{2r}-
    \sum\limits_{\substack{1\leq s \leq
    m_{i''}}} (tc_{\square_i}+b_i-ty_{i'',s}-b_{i''})^{2r}
\end{align*} where $0\leq i', i''\leq
t-1$ satisfy $i'\equiv i+k (\modsymb t)$ and $i''\equiv
i-k (\modsymb t)$. Furthermore,
\begin{align*}
\sum_{\square\in \lambda^+}h_{\square}^{2r} - \sum_{\square\in
\lambda}h_{\square}^{2r} &=t^{2r} 
    +\sum_{j=0}^{t-1}\Bigl(\sum\limits_{\substack{0\leq s \leq
    m_j}} (tc_{\square_i}+b_i-tx_{j,s}-b_j)^{2r}
    \\ &-
    \sum\limits_{\substack{1\leq s \leq
    m_j}} (tc_{\square_i}+b_i-ty_{j,s}-b_j)^{2r}\Bigr).
\end{align*}
\end{lem}

For the doubled distinct partition $\la$ 
whose image under Littlewood decomposition is $(\la_\tcore; \la^0, \la^1, \ldots, \la^{t-1})$ where $\la^0 = \sla^0 \sla^0$, let
$x_{0,s}\ (0\leq s\leq m_0)$ be the contents of inner corners of
$\sla^0$
and
$y_{0,s}\ (1\leq s\leq m_0)$ be the  contents of outer corners of
$\sla^0$.  Let
$x_{i,s}\ (0\leq s\leq m_i)$ be the contents of inner corners of
$\lambda^i$
and
$y_{i,s}\ (1\leq s\leq m_i)$ be the  contents of outer corners of
$\lambda^i$ for $1\leq i\leq t-1$. Then
$x_{i,s}=-x_{t-i,m_i-s}$ and $y_{i,s}=-y_{t-i, m_i+1-s}$ since $\lambda^i$ and $\lambda^{t-i}$ are conjugate to each other for $1\leq i \leq t-1$.
\begin{proof}[Proof of Lemma \ref{th:admissible}]
Let
$\lambda=(\sla_\tcore; \sla^0,\lambda^1,\ldots,\lambda^{t'})$ be a doubled distinct partition
and
$b_j = b_j(\sla_\tcore\sla_\tcore)$
for $0\leq j \leq t-1$.
The following statements are consequences of Lemma \ref{th:contentadd1}. 
\smallskip

$(C1)$ Let $1\leq i\leq t'$ and  
$
\lambda^+=(\sla_\tcore;
\sla^0,\lambda^1,\ldots,\lambda^{i-1},(\lambda^{i})^+,\lambda^{i+1},\ldots,\lambda^{t'})
$
be a doubled distinct partition
such that $(\lambda^{i})^+=\lambda^{i}\cup{\square_i}$,  we have
$$
\mathcal{C}(\lambda^+)\setminus \mathcal{C}(\lambda)= 
\{tc_{\square_i}+b_i-j:0\leq j\leq t-1\}\cup
\{-tc_{\square_i}+b_{t-i}-j:0\leq j\leq t-1\}.
$$

$(C2)$ Let
$
\lambda^+=(\sla_\tcore; (\sla^0)^+,\lambda^1,\ldots,\lambda^{t'})
$
be a doubled distinct partition such that $(\sla^0)^+=\sla^0\cup{\square_0}$,  we have
$$
\mathcal{C}(\lambda^+)\setminus \mathcal{C}(\lambda)= 
\{tc_{\square_0}+b_0-j:0\leq j\leq t-1\}\cup
\{t(1-c_{\square_0})+b_0-j:0\leq j\leq t-1\}.
$$
\noindent
Hence, $\sum_{c\in \setC(\lambda)}c^{r} $
is $\mu$-admissible for any nonnegative integer~$r$.

\medskip
On the other hand, we obtain the following results 
by Lemma \ref{th:hookdiff}.

$(H1)$ Let $1\leq i \leq t'$ and
$
\lambda^+=(\sla_\tcore;
\sla^0,\lambda^1,\ldots,\lambda^{i-1},(\lambda^{i})^+,\lambda^{i+1},\ldots,\lambda^{t'})
$
be a doubled distinct partition
such that
$(\lambda^{i})^+=\lambda^{i}\cup{\square_i}$, we have 

\begin{align*}
 & \sum_{\square\in \lambda^+}h_{\square}^{2r} - \sum_{\square\in
\lambda}h_{\square}^{2r} \\ &=t^{2r} +  (tc_{\square_i}+b_i-b_0)^{2r}\\&
    +\sum_{j=1}^{t-1}\Bigl(\sum\limits_{\substack{0\leq s \leq
    m_j}} (tc_{\square_i}+b_i-tx_{j,s}-b_j)^{2r}
    -
    \sum\limits_{\substack{1\leq s \leq
    m_j}} (tc_{\square_i}+b_i-ty_{j,s}-b_j)^{2r}\Bigr)
    \\&+
    \sum\limits_{\substack{1\leq s \leq
    m_j}} (tc_{\square_i}+b_i-tx_{0,s}-b_0)^{2r}
    -
    \sum\limits_{\substack{1\leq s \leq
    m_j}} (tc_{\square_i}+b_i-ty_{0,s}-b_0)^{2r}
    \\&+
    \sum\limits_{\substack{1\leq s \leq
    m_j}} (tc_{\square_i}+b_i-t(1-x_{0,s})-b_0)^{2r}
    -
    \sum\limits_{\substack{1\leq s \leq
    m_j}} (tc_{\square_i}+b_i-t(1-y_0,s)-b_0)^{2r} 
    \\ &+t^{2r} +  (-tc_{\square_i}+b_{t-i}-b_0)^{2r}\\&
    +\sum_{j=1}^{t-1}\Bigl(\sum\limits_{\substack{0\leq s \leq
    m_j}} (-tc_{\square_i}+b_{t-i}-tx_{j,s}-b_j)^{2r}
    -
    \sum\limits_{\substack{1\leq s \leq
    m_j}} (-tc_{\square_i}+b_{t-i}-ty_{j,s}-b_j)^{2r}\Bigr)
    \\&+
    \sum\limits_{\substack{1\leq s \leq
    m_j}} (-tc_{\square_i}+b_{t-i}-tx_{0,s}-b_0)^{2r}
    -
    \sum\limits_{\substack{1\leq s \leq
    m_j}} (-tc_{\square_i}+b_{t-i}-ty_{0,s}-b_0)^{2r}
    \\&+
    \sum\limits_{\substack{1\leq s \leq
    m_j}} (-tc_{\square_i}+b_{t-i}-t(1-x_{0,s})-b_0)^{2r}
    \\&-
    \sum\limits_{\substack{1\leq s \leq
    m_j}} (-tc_{\square_i}+b_{t-i}-t(1-y_0,s)-b_0)^{2r} 
     + (-tc_{\square_i}+b_{t-i}-t(c_{\square_i}+1)-b_0)^{2r}
    \\& + (-tc_{\square_i}+b_{t-i}-t(c_{\square_i}-1)-b_0)^{2r}
      -2  (-tc_{\square_i}+b_{t-i}-tc_{\square_i}-b_0)^{2r}.
\end{align*}

$(H2)$ Let
$
\lambda^+=(\sla_\tcore;
(\sla^0)^+,\lambda^1,\ldots,\lambda^{t'})
$
be a doubled distinct partition
such that $(\sla^0)^+=\sla^0\cup{\square_0}$,  we have
\begin{align*}
 & \sum_{\square\in \lambda^+}h_{\square}^{2r} - \sum_{\square\in
\lambda}h_{\square}^{2r} \\ &=t^{2r} +  (tc_{\square_0})^{2r}\\&
    +\sum_{j=1}^{t-1}\Bigl(\sum\limits_{\substack{0\leq s \leq
    m_j}} (tc_{\square_0}+b_0-tx_{j,s}-b_j)^{2r}
    -
    \sum\limits_{\substack{1\leq s \leq
    m_j}} (tc_{\square_0}+b_0-ty_{j,s}-b_j)^{2r}\Bigr)
    \\&+
    \sum\limits_{\substack{1\leq s \leq
    m_j}} (tc_{\square_0}-tx_{0,s})^{2r}
    -
    \sum\limits_{\substack{1\leq s \leq
    m_j}} (tc_{\square_0}-ty_{0,s})^{2r}
    \\&+
    \sum\limits_{\substack{1\leq s \leq
    m_j}} (tc_{\square_0}-t(1-x_{0,s}))^{2r}
    -
    \sum\limits_{\substack{1\leq s \leq
    m_j}} (tc_{\square_0}-t(1-y_0,s))^{2r} 
    \\ &+t^{2r} +  (t-tc_{\square_0})^{2r}\\&
    +\sum_{j=1}^{t-1}\Bigl(\sum\limits_{\substack{0\leq s \leq
    m_j}} (t-tc_{\square_0}+b_0-tx_{j,s}-b_j)^{2r}
    -
    \sum\limits_{\substack{1\leq s \leq
    m_j}} (t-tc_{\square_0}+b_0-ty_{j,s}-b_j)^{2r}\Bigr)
    \\&+
    \sum\limits_{\substack{1\leq s \leq
    m_j}} (t-tc_{\square_0}-tx_{0,s})^{2r}
    -
    \sum\limits_{\substack{1\leq s \leq
    m_j}} (t-tc_{\square_0}-ty_{0,s})^{2r}
    \\&+
    \sum\limits_{\substack{1\leq s \leq
    m_j}} (t-tc_{\square_0}-t(1-x_{0,s}))^{2r}
    -
    \sum\limits_{\substack{1\leq s \leq
    m_j}} (t-tc_{\square_0}-t(1-y_0,s))^{2r} 
    \\& + (t-tc_{\square_0}-t(c_{\square_0}+1))^{2r}
     + (t-tc_{\square_0}-t(c_{\square_0}-1))^{2r}
     \\& -2  (t-tc_{\square_0}-tc_{\square_0})^{2r}.
\end{align*}
Hence, $\sum_{h\in \setH(\lambda)}h^{2r} $
is $\mu$-admissible for any nonnegative integer~$r$.
\qedhere
\end{proof}


\subsection{Main results for doubled distinct partitions}\label{maindoubled distinct} 
To prove the doubled distinct case of Theorem \ref{th:intro},
we establish the following more general result.

\begin{thm} \label{th:main}
	Let $(\nu^0, \nu^1, \ldots, \nu^{t'})$ be a $(t'+1)$-tuple of usual partitions, and $\alpha$ be a $t$-core doubled distinct partition.
Suppose that $g_1, g_2, \ldots, g_v$ are $\alpha$-admissible functions of
doubled distinct partitions. 
Then, there exists some
$r\in \mathbb{N}$ such that
$$
D_t^r\Bigl(G_\lambda \prod_{u=1}^{v}g_u(\lambda)\Phi^{\nu^0}(\sla^0)\prod_{i=1}^{t'}\Psi^{\nu^i}(\lambda^i) \Bigr)=0
$$
for every doubled distinct partition $\lambda$ with $\lambda_\tcore=\alpha.$
Furthermore, let $\mu$ be a given doubled distinct partition. By Theorem \ref{th:main3}, 
\begin{equation}
\sum\limits_{\substack{\la\in\setDD,\  \lambda\geq_t\mu\\
|\lambda/\mu|=2nt }} F_{\lambda/\mu} G_\lambda\, \prod_{u=1}^{v}g_u(\lambda)
\end{equation}
is a polynomial in $n$. 
\end{thm}
\begin{proof}
   We will prove this claim by induction.
   Let
\begin{align*}
	A&= \prod_{u=1}^{v}g_u(\lambda),\qquad
	B= \prod_{i=1}^{t'}\Psi^{\nu^i}(\lambda^i),\qquad
    C= \Phi^{\nu^0}(\sla^0),
	\\
	\Delta A&= \prod_{u=1}^{v}g_u(\rho)-\prod_{u=1}^{v}g_u(\lambda)= \sum\limits_{(*)}\prod_{s\in U}g_s(\lambda)
{\prod_{s'\in V} \bigl(g_{s'}(\rho)-g_{s'}(\lambda)\bigr)},\\
	\Delta B&= 
\prod_{i=1}^{t'}\Psi^{\nu^i}(\rho^i)-\prod_{i=1}^{t'}\Psi^{\nu^i}(\la^i)\\&
= \sum\limits_{(**)}\prod_{s\in U}\Psi^{\nu^s}(\la^i)
{\prod_{s'\in V} \bigl(\Psi^{\nu^{s'}}(\rho^{i})-\Psi^{\nu^{s'}}(\la^i)\bigr)},\\
	\Delta C&= 
\Phi^{\nu^0}({\bar \rho^0})-\Phi^{\nu^0}(\sla^0),
\end{align*}
where the sum $(*)$ (resp. $(**)$) ranges over all
pairs $(U,V)$ of positive integer sets such that 
$U\cup V=\{1,2,\ldots,v\}$
(resp. $U\cup V=\{1,2,\ldots,t'\}$), 
$U\cap V=\emptyset$ and $V\neq \emptyset$.
    We have
\begin{align*}
&D_t\Bigl(G_\lambda \prod_{u=1}^{v}g_u(\lambda)\Phi^{\nu^0}(\sla^0)\prod_{i=1}^{t'}\Psi^{\nu^i}(\lambda^i) \Bigl)  \\ 
=\ &G_\lambda\sum\limits_{\substack{\rho\geq_t\lambda\\
|\rho/\lambda|=2t}}\frac{G_{\rho}}{G_\lambda} \Bigl(
\prod_{u=1}^{v}g_u(\rho)\Phi^{\nu^0}({\bar\rho}^0)\prod_{i=1}^{t'}\Psi^{\nu^i}(\rho^i)\\ 
 &\quad -
\prod_{u=1}^{v}g_u(\lambda)\Phi^{\nu^0}(\sla^0)\prod_{i=1}^{t'}\Psi^{\nu^i}(\lambda^i)
\Bigr)\\
=\ &G_\lambda\sum\limits_{\substack{\rho\geq_t\lambda\\
|\rho/\lambda|=2t}}\frac{G_{\rho}}{G_\lambda} \bigl( 
 \Delta A\cdot   B \cdot  C
+ A\cdot  \Delta B \cdot  C
+A\cdot   B \cdot \Delta C
 \\
&\quad 
+ A\cdot \Delta  B \cdot \Delta C
+ \Delta A\cdot   B \cdot \Delta C 
+  \Delta A\cdot  \Delta B \cdot  C  
+ \Delta A\cdot  \Delta B \cdot \Delta C 
\bigr).
\end{align*}
For the first term in the above summation,
 we obtain
\begin{align*}	&G_\lambda\sum\limits_{\substack{\rho\geq_t\lambda\\
|\rho/\lambda|=2t}}\frac{G_{\rho}}{G_\lambda} \bigl(
\Delta A\cdot    B \cdot C
\bigr)\\
&= \frac1t G_{\lambda}
\Phi^{\nu^0}(\sla^0)\prod_{i=1}^{t'}\Psi^{\nu^i}(\lambda^i) \sum_{0\leq i\leq m_0} \frac{\prod\limits_{\substack{1\leq
j\leq m_0}}\bigl({\binom{x_{0,i}}{2}}-{\binom{y_{0,j}}{2}}\bigr)}{\prod\limits_{\substack{0\leq j\leq m_0\\
j\neq i}}\bigl({\binom{x_{0,i}}{2}}-{\binom{x_{0,j}}{2}}\bigr)}\\
&\qquad\times
\sum\limits_{(*)}\prod_{s\in U}g_s(\lambda)
{\prod_{s'\in V} \bigl(g_{s'}((\smu;
(\sla^0)^{i+},\lambda^1,\ldots,\lambda^{t'}))-g_{s'}(\lambda)\bigr)}\\
&+ 
\frac2t G_{\lambda} \Phi^{\nu^0}(\sla^0)\prod_{i=1}^{t'}\Psi^{\nu^i}(\lambda^i) 
\sum_{k=1}^{t'} 
\sum_{0\leq i\leq m_k} \frac{\prod\limits_{\substack{1\leq
j\leq m_k}}(x_{k,i}-y_{k,j})}{\prod\limits_{\substack{0\leq j\leq m_k\\
j\neq i}}(x_{k,i}-x_{k,j})}\\
&\qquad \times
\sum\limits_{(*)}\prod_{s\in U}g_s(\lambda)
{\prod_{s'\in V} \bigl(g_{s'}((\smu;
\sla^0,\lambda^1,\ldots,({\lambda^{k}})^{i+},\ldots,\lambda^{t'}))-g_{s'}(\lambda)\bigr)}
\end{align*}
where $\setC((\sla^0)^{i+})\setminus \setC(\sla^0)=x_{0,i}$ and $\setC(({\lambda^{k}})^{i+})\setminus \setC(\lambda^{k})=x_{k,i}.$
Since $g_1, g_2, \ldots, g_v$ are $\alpha$-admissible functions and thanks to 
Lemma \ref{th:aibi},
$$G_\lambda\sum\limits_{\substack{\rho\geq_t\lambda\\
|\rho/\lambda|=2t}}\frac{G_{\rho}}{G_\lambda} \bigl(
\Delta A\cdot    B \cdot C
\bigr)$$  could
be written as a linear combination of some 
$$G_\lambda \prod_{\underline u=1}^{{\underline v}}
g_{\underline u}( \lambda)
\Phi^{{\underline \nu}^0}({\bar\lambda}^0)\prod_{i=1}^{t'}\Psi^{{\underline \nu}^i}({ \lambda}^i)$$
where either $\underline v<v$, or $\underline v=v$ and simultaneously  
$$2|{\underline \nu}^0|+\sum_{i=1}^{t'}|{\underline \nu}^i|\leq 
2|{\nu}^0| + \sum_{i=1}^{t'}|{\nu}^i|-2.$$ 
In the other hand,
we have similar results
for other six terms by Lemmas  \ref{th:XY}, \ref{th:aibi} and Theorem 
\ref{th:diffq1}. 
Thus, Theorem \ref{th:main} is proved by induction on $(v, 2|{\nu}^0| + \sum_{i=1}^{t'}|{\nu}^i|)$.
\end{proof}

As an application of Theorem \ref{th:main},
we derive the doubled distinct case of Theorem~\ref{th:intro} from Lemma \ref{th:admissible}.
Actually, by a similar but more precise argument as in the proof of Lemma \ref{th:admissible}, we can show that 
$$
\sum\limits_{\substack{\square\in \lambda\\
h_{\square}\equiv \pm j (\modsymb t )}}h_{\square}^{2r}
\quad
\text{and}
\quad
\sum\limits_{\substack{\square\in \lambda\\
c_{\square}\equiv j  (\modsymb t )}}c_{\square}^{r}
$$ 
are $\mu$-admissible for 
any $t$-core doubled distinct partition $\mu$,  any nonnegative integer~$r$, and $0\leq j\leq t-1$.
By Theorem \ref{th:main} we derive the following result.

\begin{cor} \label{th:main''}
Let $u',v', j_u,j'_v, k_u,k'_v$ be nonnegative integers and $\alpha$ be a given $t$-core doubled distinct
partition. Then, there exists some $r\in \mathbb{N}$ such that
$$
D_t^r\Biggl(G_\lambda \Biggl(\prod_{u=1}^{u'}\sum\limits_{\substack{\square\in \lambda\\
h_{\square}\equiv \pm j_u (\modsymb t
)}}h_{\square}^{2k_u}\Biggr)\Biggl(\prod_{v=1}^{v'}\sum\limits_{\substack{\square\in \lambda\\
c_{\square}\equiv j'_v (\modsymb t )}}c_{\square}^{k'_v}\Biggr) \Biggr)=0
$$
for every doubled distinct partition $\lambda$ with $\lambda_\tcore=\alpha.$
Moreover, let $\mu$ be a given doubled distinct partition.
\begin{equation*}
\sum\limits_{\substack{\la\in\setDD,\  \lambda\geq_t\mu\\
|\lambda/\mu|=2nt }} F_{\lambda/\mu} G_\lambda\,
\Biggl(\prod_{u=1}^{u'}\sum\limits_{\substack{\square\in \lambda\\
h_{\square}\equiv \pm j_u  (\modsymb t
)}}h_{\square}^{2k_u}\Biggr)\Biggl(\prod_{v=1}^{v'}\sum\limits_{\substack{\square\in \lambda\\
c_{\square}\equiv j'_v (\modsymb t )}}c_{\square}^{k'_v}\Biggr)
\end{equation*}
is a polynomial in $n$ of degree at most $\sum_{u=1}^{u'}(k_u+1)+\sum_{v=1}^{v'}\frac{k'_v+2}{2}$.
\end{cor}


\goodbreak

\section{Polynomiality for self-conjugate partitions}\label{mainSC} 

In this section we always set that $t=2t'$ is an even integer.  
The set of all $t$-core self-conjugate partitions is denoted by $\setSC_\tcore$. 
Let $\la$ be a self-conjugate partition. By \cite{stanton}, the Littlewood decomposition maps 
$\la$ to
$$(\la_{\tcore}; \lambda^0, \lambda^1,\ldots, \lambda^{t-1})\in 
\setSC_\tcore\times  \setP^{2t'}$$ 
where $\lambda^i$ is the conjugate partition of $\lambda^{t-1-i}$ 
for $0\leq i\leq t'-1$.
 For convenience, we always write
$$\lambda=(\lambda_\tcore;
\lambda^0,\ldots,\lambda^{t'-1}).$$
Let
$\lambda=(\lambda_\tcore; \lambda^0,\ldots,\lambda^{t'-1})$
and
$\mu=(\mu_\tcore; \mu^0,\ldots,\mu^{t'-1})$
be 
two self-conjugate partitions.
If  $\lambda_\tcore=\mu_\tcore$ and $\lambda^i\supset \mu^i$ for $0\leq i\leq t'-1$, 
we write 
$\lambda\geq_t \mu$  and define
\begin{equation}
F_{\mu/\mu}:=1
\qquad \text{and}\qquad F_{\lambda/\mu}:=\sum\limits_{\substack{\lambda\geq_t\lambda^-\geq_t\mu\\
|\lambda/\lambda^-|=2t }}F_{\lambda^-/\mu}\qquad \text{(for $\lambda \neq \mu$)}.
\end{equation}
Then
$F_{\lambda/\mu}$
is the number of vectors 
$(P_0,P_1,\ldots,P_{t'-1})$
such that

(1) $P_i$ $(0\leq i\leq t'-1)$  {is a skew
Young tableau of shape}  $\lambda^i/\mu^i$, 

(2) {the union of
entries in} $P_0,P_1,\ldots,P_{t'}$ \emph{are} $1,\ 2,\ldots,\ 
n=\sum_{i=0}^{t'-1}|\lambda^i/\mu^i|$.

\noindent
Hence,
$$F_{\lambda/\mu}=\binom{\sum_{i=0}^{t'-1}|\lambda^i/\mu^i|}
{|\lambda^0/\mu^0|,\ldots,|\lambda^{t'}/\mu^{t'}|}\prod_{i=0}^{t'-1}
 f_{\lambda^i/\mu^i}.$$
We set
\begin{equation}
F_\lambda:=
F_{\lambda/ \lambda_{\tcore}}=
\binom{\sum_{i=0}^{t'-1}|\lambda^i|}
{|\lambda^0|,\ldots,|\lambda^{t'-1}|}\prod_{i=0}^{t'-1}
 f_{\lambda^i}
 =\frac {n!}{\prod_{i=0}^{t'-1} H(\la^i)}
 \end{equation}
 and
$$G_{\lambda}:=\frac {2^{n}}{t^n\prod_{i=0}^{t'-1} H(\lambda^i)}
=\frac {2^n F_\la}{t^n n!}.
$$

Let $g$ be a function of self-conjugate partitions and $\lambda$ be a self-conjugate partition.
The \emph{$t$-difference operator~$D_t$} for self-conjugate partitions is defined  by
\begin{equation}\label{eq:diffDtdoubled distinct}  D_tg(\lambda)=\sum\limits_{\substack{\lambda^{+}\geq_t\lambda\\ |\lambda^{+}/\lambda|=2t}}g(\lambda^{+})-g(\lambda).
\end{equation}
The higher-order $t$-difference operators $D_t^k$ are defined by induction:
$$D_t^0g:=g \text{\quad and\quad }\ D_t^k g:=D_t(D_t^{k-1} g)\quad (k\geq 1).$$

\begin{lem} \label{th:SCGlambda}
Suppose that $\lambda$ is a self-conjugate partition. Then
$
D_t (G_\lambda)=0
$.
In other words,
\begin{equation}\label{eq:Gsc}
G_{\lambda}=\sum\limits_{\substack{\lambda^+\geq_t\lambda\\
|\lambda^+/\lambda|=2t }}G_{\lambda^+}.
\end{equation}

\end{lem}
\begin{proof}
Write 
$\lambda=(\lambda_\tcore; \lambda^0,\ldots,\lambda^{t'-1})$.
For $0\leq i\leq t'-1$ we obtain
$$
\sum_{|(\lambda^{i})^+/\lambda^{i}|=1}
\frac{G_{(\lambda_\tcore;
\lambda^0,\ldots,\lambda^{i-1},(\lambda^{i})^+,\lambda^{i+1},\ldots,\lambda^{t'-1})}}{G_\lambda}
=
\sum_{|(\lambda^{i})^+/\lambda^{i}|=1}
\frac{2H(\lambda^{i})}{tH((\lambda^{i})^+)}=\frac{2}{t}
$$
by Lemma 2.2 in \cite{hanxiong1}. 
Summing the above equalities we prove \eqref{eq:Gsc}.
\end{proof}
 
By analogy with the results on doubled distinct partitions, we have the following theorems for self-conjugate partitions. Their proofs are omitted.

\begin{lem} \label{th:SCtelescope}
Suppose that $\mu$ is a  given self-conjugate partition and $g$ is a function    of self-conjugate
 partitions. For every nonnegative  integer $n$, let
$$
P(\mu,g; n):=\sum\limits_{\substack{\la\in\setSC,\ \lambda\geq_t\mu\\
|\lambda/\mu|=2nt }}F_{\lambda/\mu}g(\lambda).
$$
Then 
$$
P(\mu,g; n+1)-P(\mu, g; n)=P(\mu, {D_tg}; n).
$$
\end{lem}

{\it Example}. Let  $g(\lambda)=G_\lambda$. Then $D_tg(\lambda)=0$
by Lemma \ref{th:SCGlambda}, which means that 
\begin{equation}
\sum\limits_{\substack{\la\in\setSC,\  \lambda\geq_t\mu\\
|\lambda/\mu|=2nt }}F_{\lambda/\mu} G_\lambda= G_\mu.
\end{equation}
When $\mu=\emptyset$, the above identity becomes
\begin{equation}
\sum\limits_{\substack{\lambda \in \mathcal{SC}, \ 
|\lambda|=2nt \\ \lambda_\tcore= \emptyset } } \frac{(2t)^n n!}{\prod_{h\in\mathcal{H}_t(\lambda)}h}\,=1.
\end{equation}

\medskip

\begin{thm}\label{th:SCmain3}
Let $g$ be a function of  self-conjugate partitions and $\mu$ be a given self-conjugate partition.
Then, 
\begin{equation}\label{eq:main3}
P(\mu, g; n)=\sum\limits_{\substack{\la\in\setSC,\ \lambda\geq_t\mu\\
|\lambda/\mu|=2nt }}F_{\lambda/\mu}g(\lambda)
=\sum_{k=0}^n\binom{n}{k}D_t^kg(\mu)
\end{equation}
and
\begin{equation}\label{eq:main3b}
D_t^ng(\mu)=\sum_{k=0}^n(-1)^{n+k}\binom{n}{k}P(\mu, g; k).
\end{equation}
In particular, if there exists some positive integer $r$ such that
$D_t^r g(\lambda)=0$ for every self-conjugate partition $\lambda\geq_t \mu$, then $P(\mu, g; n)$ is
a polynomial in $n$ of degree at most $r-1$.
\end{thm}

\begin{thm} \label{th:mainSC} Let $t=2t'$ be a given integer, $\alpha$ be a given $t$-core self-conjugate
partition, and $u',v', j_u,{j'}_v, k_u,{k'}_v$ be nonnegative integers.
Then there exists some
$r\in \mathbb{N}$ such that
$$
D_t^r\left(\prod_{u=1}^{u'}\sum\limits_{\substack{\square\in \lambda\\
h_{\square}\equiv \pm j_u  (\modsymb t
)}}h_{\square}^{2k_u}\Biggr)\Biggl(\prod_{v=1}^{v'}\sum\limits_{\substack{\square\in \lambda\\
c_{\square}\equiv j'_v (\modsymb t )}}c_{\square}^{k'_v} \right)=0
$$
for every self-conjugate partition $\lambda$ with $\lambda_\tcore=\alpha$.  
Furthermore,  let $\mu$ be a given self-conjugate
partition. Then  by Theorem \ref{th:SCmain3}, we have

\begin{equation*}
\sum\limits_{\substack{\la\in\setSC,\ \lambda\geq_t  \mu\\
|\lambda/\mu|=2nt }} F_{\lambda/\mu} G_\lambda\,
\Biggl(\prod_{u=1}^{u'}\sum\limits_{\substack{\square\in \lambda\\
h_{\square}\equiv \pm j_u  (\modsymb t
)}}h_{\square}^{2k_u}\Biggr)\Biggl(\prod_{v=1}^{v'}\sum\limits_{\substack{\square\in \lambda\\
c_{\square}\equiv j'_v (\modsymb t )}}c_{\square}^{k'_v}\Biggr)
\end{equation*}
is a polynomial in $n$ of degree at most $\sum_{u=1}^{u'}(k_u+1)+\sum_{v=1}^{v'}\frac{k'_v+2}{2}$. 
\end{thm}

\section{Square cases for doubled distinct and self-conjugate partitions}
As described in Corollary \ref{th:DDSCsquare}, the polynomials mentioned
in Corollary \ref{th:main''} and Theorem \ref{th:mainSC} have explicit expressions for square cases.

\begin{proof}[Proof of  Corollary \ref{th:DDSCsquare}]

	(1) When $\la$ is a doubled distinct partition with $|\lambda|=2nt$ ($t$ odd) and $\lambda_\tcore=\emptyset$. By the proof of Lemma \ref{th:admissible}  we obtain
\begin{align*}
\frac{1}{G_\lambda}D\bigl(G_\lambda\, (\sum\limits_{\substack{\square\in \lambda}}c_{\square}^{2})\bigr)
&=\frac1t \sum_{0\leq i\leq m_0} \frac{\prod\limits_{\substack{1\leq
j\leq m_0}}\bigl({\binom{x_{0,i}}{2}}-{\binom{y_{0,j}}{2}}\bigr)}{\prod\limits_{\substack{0\leq j\leq m_0\\
j\neq i}}\bigl({\binom{x_{0,i}}{2}}-{\binom{x_{0,j}}{2}}\bigr)}
\\&\qquad \times \sum_{j=0}^{t-1}\bigl((tx_{0,i}-j)^2+(t-tx_{0,i}-j)^2\bigr)
\\& + \frac2t \sum_{1\leq k\leq t'} \sum_{0\leq i\leq m_k} \frac{\prod\limits_{\substack{1\leq
j\leq m_k}}(x_{k,i}-y_{k,j})}{\prod\limits_{\substack{0\leq j\leq m_k\\
j\neq i}}(x_{k,i}-x_{k,j})}\\
&\qquad \times \sum_{j=0}^{t-1}\bigl((tx_{k,i}+k-j)^2+(-tx_{k,i}+t-k-j)^2\bigr)
\\&=
\frac1t \sum_{0\leq i\leq m_0} \frac{\prod\limits_{\substack{1\leq
j\leq m_0}}\bigl({\binom{x_{0,i}}{2}}-{\binom{y_{0,j}}{2}}\bigr)}{\prod\limits_{\substack{0\leq j\leq m_0\\
j\neq i}}\bigl({\binom{x_{0,i}}{2}}-{\binom{x_{0,j}}{2}}\bigr)}
\\&\qquad \times  \bigl(4t^3\binom{x_{0,i}}{2}+t^3-t^2(t-1)+\frac{(t-1)t(2t-1)}{3} \bigr)
\\& + \frac2t \sum_{1\leq k\leq t'} \sum_{0\leq i\leq m_k} \frac{\prod\limits_{\substack{1\leq
j\leq m_k}}(x_{k,i}-y_{k,j})}{\prod\limits_{\substack{0\leq j\leq m_k\\
j\neq i}}(x_{k,i}-x_{k,j})}
\\
&\qquad \times \bigl(2t^3{x^2_{k,i}}+\sum_{j=0}^{t-1}(k-j)^2+\sum_{j=0}^{t-1}(t-k-j)^2 \bigr)
\\& = 2t|\lambda|+\frac{t(t^2+2)}{3},
\end{align*}
therefore
$$
\frac{1}{G_\lambda}D^2\bigl(G_\lambda\, (\sum\limits_{\substack{\square\in \lambda}}c_{\square}^{2})\bigr)=4t^2,
$$
and
$$
\frac{1}{G_\lambda}D^3\bigl(G_\lambda\, (\sum\limits_{\substack{\square\in \lambda}}c_{\square}^{2})\bigr)=0.
$$

(2)  When $\la$ is a self-conjugate partition with $|\lambda|=2nt$ ($t$ even) and $\lambda_\tcore=\emptyset$.  Similarly as in (1) we have
$$\frac{1}{G_\lambda}D\bigl(G_\lambda\, (\sum\limits_{\substack{\square\in \lambda}}c_{\square}^{2})\bigr)=2t|\lambda|+\frac{t(t^2-1)}{3},
$$
$$\frac{1}{G_\lambda}D^2\bigl(G_\lambda\, (\sum\limits_{\substack{\square\in \lambda}}c_{\square}^{2})\bigr)=4t^2,
$$
and
$$\frac{1}{G_\lambda}D^3\bigl(G_\lambda\, (\sum\limits_{\substack{\square\in \lambda}}c_{\square}^{2})\bigr)=0.
$$
Then identities \eqref{eq:DDc2} and \eqref{eq:SCc2}  follows from Theorems \ref{th:main3} and \ref{th:SCmain3}.
Notice that  $\sum_{\square\in \lambda}h_{\square}^{2}-\sum_{{\square\in \lambda}}c_{\square}^{2}=|\lambda|^2$ (see \cite{Macdonald}).
Identities~\eqref{eq:DDh2} and \eqref{eq:SCh2}  are consequences
of identities \eqref{eq:DDc2} and \eqref{eq:SCc2}.
\end{proof}

\section{Acknowledgments}
The second author
is supported  by grant [PP00P2\_138906] of the Swiss National Science Foundation and  Forschungskredit [FK-14-093] of the
University of Zurich. He also  thanks  Prof.
P.-O. Dehaye for the encouragements and helpful suggestions.



\bigskip
\bigskip
\bigskip


\begin{thebibliography}{1} 

\bibitem{AF} R. M. Adin and A. Frumkin, Rim hook tableaux and Kostant's $\eta$-function coefficients, \emph{Adv. Appl. Math.} \textbf{33(3)}(2004), 492--511.


\bibitem{bandlow} J. Bandlow, An elementary proof of the Hook
formula, \emph{Electron. J. Combin.} \textbf{15} (2008), research
paper 45.

\bibitem{hanxiong2} P.-O. Dehaye, G.-N. Han, and H. Xiong, Difference operators for partitions under the Littlewood decomposition, preprint; {\tt arXiv:1511.02804}.

\bibitem{fkmo}  S. Fujii, H. Kanno, S. Moriyama, and S. Okada,
     Instanton calculus and chiral one-point functions
     in supersymmetric gauge theories,
     \emph{Adv.\ Theor.\ Math.\ Phys.} \textbf{12} (2008), no. 6, 1401--1428.

\bibitem{stanton} F. Garvan, D. Kim, and D. Stanton,  Cranks and t-cores, \emph{Invent. Math.} \textbf{101} (1990), 1--18.


\bibitem{han2} G.-N. Han, The Nekrasov-Okounkov hook length formula:
  refinement, elementary proof, extension, and applications, \emph{Ann. Inst. Fourier} \textbf{60(1)} (2010),  1--29.

\bibitem{han} G.-N. Han, Some conjectures and open problems on
  partition hook lengths, \emph{Experimental Mathematics} \textbf{18} (2009),  97--106.

\bibitem{han3} G.-N. Han, Hook lengths and shifted parts of
  partitions, \emph{Ramanujan J.}
\textbf{23(1-3)} (2010),  127--135.

\bibitem{han4} G.-N. Han and K. Q. Ji, Combining hook length formulas and BG-ranks for partitions via the
 Littlewood decomposition, \emph{Trans. Amer. Math. Soc.}
  \textbf{363} (2011),  1041--1060.
  
  \bibitem{hanxiong1} G.-N.~Han and H.~Xiong, Difference operators
for partitions and some applications, preprint; {\tt arXiv:1508.00772}.

\bibitem{hanxiong3} G.-N.~Han and H.~Xiong, New hook-content formulas for strict partitions,  {\tt arXiv:1511.02829}.


\bibitem{JK} G. James and A. Kerber, \emph{The representation theory of the symmetric group,  Encyclopedia of Mathematics and its Applications, 16}, Addison-Wesley
Publishing, Reading, MA, 1981.


\bibitem{rsk} D. Knuth, \emph{The Art of Computer Programming, Vol. 3: Sorting and Searching}, Addison--Wesley,
London,
 1973, pp. 54--58.

\bibitem{Lascoux} A. Lascoux, \emph{Symmetric functions and combinatorial operators on polynomials},  Vol. 99. American Mathematical Soc., 2003.

\bibitem{Macdonald} I. G. Macdonald, \emph{Symmetric functions and Hall polynomials}, Oxford Mathematical Monographs,
The Clarendon Press, Oxford University Press, New York, second edition,
1995.

\bibitem{no} N. A. Nekrasov and A. Okounkov, Seiberg-Witten theory
  and random partitions, \emph{The unity of mathematics}, \emph{Progress
    in Mathematics} \textbf{244}, Birkh\"auser Boston, 2006,
  pp. 525--596.



\bibitem{ols} G. Olshanski,
Anisotropic Young diagrams and infinite-dimensional diffusion
processes with the Jack parameter, \emph{Int. Math. Res. Not. IMRN}
\textbf{6} (2010),  1102--1166.

\bibitem{ols2} G. Olshanski, Plancherel averages: Remarks on a paper by Stanley,
 \emph{Electron. J. Combin.} \textbf{17} (2010), research paper 43.

\bibitem{panova} G. Panova,
Polynomiality of some hook-length statistics, \emph{Ramanujan J.}
\textbf{27(3)} (2012),  349--356.

\bibitem{petreolle1} M. P\'etr\'eolle, A Nekrasov-Okounkov type formula for $\widetilde C$, {\tt arXiv:1505.01295}.

\bibitem{petreolle2} M. P\'etr\'eolle, \emph{Quelques d\'eveloppements combinatoires autour des groupes de Coxeter et des partitions d'entiers}, Ph.D. thesis, 2015.

\bibitem{Rota1964}
G.-C. Rota,
On the foundations of combinatorial theory: {I}. {T}heory of
{M}\"obius functions,
\emph{Z. Wahrscheinlichkeitstheorie und Verw. Gebiete}
\textbf{2} (1964), 349--356.

\bibitem{schur} I. Schur, \"Uber die Darstellung der symmetrischen und der alternienden Gruppe durch gebrochene lineare Substitutionen, \emph{J. Reine Angew. Math.} \textbf{139} (1911), 155--250.

\bibitem{stan} R. P. Stanley. Some combinatorial properties of hook lengths, contents,
and parts of partitions, \emph{Ramanujan J.} \textbf{23(1-3)}
(2010), 91--105.

\bibitem{stan1}  R. P. Stanley, Differential posets, \emph{J. Amer. Math. Soc.} \textbf{1(4)} (1988), 
919--961.



\bibitem{ec2} R. P. Stanley, \emph{Enumerative Combinatorics}, vol.~2,
  Cambridge University Press, New York/Cam\-bridge, 1999.

\bibitem{thr} R. M. Thrall, A combinatorial problem, \emph{Michigan Math. J.} \textbf{1} (1952), 81--88.


\bibitem{westbury}	B. W. Westbury, Universal characters from the Macdonald identities, \emph{Adv. Math.}, \textbf{202(1)} (2006), 50-–63.

\end{thebibliography}
\end{document}